\documentclass[twocolumn,journal]{IEEEtran}
\usepackage[utf8]{inputenc}
\usepackage[english]{babel}
\usepackage[T1]{fontenc}
\usepackage{amsmath}
\usepackage{hyperref}

\usepackage{tikz}
\usepackage{lipsum}
\usepackage{amsmath,amssymb,array}
\usepackage{siunitx}
\usepackage{longtable,tabularx}
\usepackage[ruled,lined]{algorithm2e}
\usepackage{algorithmic}
\usepackage{multirow}
\usepackage{color}
\usepackage{bm}
\usepackage{subfigure}
\usepackage{float}

\newtheorem{theorem}{Theorem}
\newtheorem{definition}{Definition}
\newtheorem{lemma}[theorem]{Lemma}%
\newtheorem{remark}{Remark}%
\newtheorem{example}{Example}%

\newcommand{\argmin}{\mathop{\rm argmin}\limits}

\def\b\phi{{\mathbb{\phi}}}

\def\int{\mathop{\rm int}}

\def\Tr{\mathop{\rm Tr}}

\def\cX{\mathcal X}

\def\X{{\mathcal{X}}}
\def\Y{{\mathcal{Y}}}

\def\cH{{\mathcal{H}}}

\def\b{{\beta}}

\def\Pro{\mathop{\Gamma}\nolimits}

\def\QED{\mbox{\rule[0pt]{1.5ex}{1.5ex}}}

\newcommand{\qed}{\hfill \QED}

 \newenvironment{proofof}[1]{\vspace*{5mm} \par \noindent
{\it Proof of #1:\hspace{2mm}}}{\qed
}

\newcommand{\nc}{\newcommand}
\nc{\ketbra}[2]{|#1\rangle\!\langle#2|}
\nc{\bra}[1]{\langle#1|}
\nc{\ket}[1]{|#1\rangle}

\def\Label#1{\label{#1}\ [\ \text{#1}\ ]\ }
\def\Label{\label}
\newcommand{\red}[1]{{\color{red} #1}}

\begin{document}

\title{Bregman-divergence-based Arimoto-Blahut algorithm}
\author{%
Masahito~Hayashi 
\thanks{The work of MH was supported 
in part by the National Natural Science Foundation of China (Grant No. 62171212)
and 
the General R\&D Projects of 
1+1+1 CUHK-CUHK(SZ)-GDST Joint Collaboration Fund 
(Grant No. GRDP2025-022).}
\thanks{
Masahito Hayashi is with 
School of Data Science, The Chinese University of Hong Kong, Shenzhen, Longgang District, Shenzhen, 518172, China,
International Quantum Academy (IQA), Futian District, Shenzhen 518048, China,
and
Graduate School of Mathematics, Nagoya University, Chikusa-ku, Nagoya 464-8602, Japan.
(e-mail: hmasahito@cuhk.edu.cn).}}
\markboth{Bregman-divergence-based Arimoto-Blahut algorithm}{}

\maketitle

\begin{abstract}
We generalize the generalized Arimoto-Blahut algorithm to a general function defined over Bregman-divergence system.
In existing methods, when linear constraints are imposed,
each iteration needs to solve a convex minimization.
Exploiting our obtained algorithm, 
we propose a minimization-free-iteration algorithm.
This algorithm can be applied to 
classical and quantum rate-distortion theory.
We numerically apply our method to the derivation of 
the optimal conditional distribution in the rate-distortion theory.
\end{abstract}

\begin{IEEEkeywords} 
Bregman divergence,
rate-distortion,
em-algorithm,
mixture family, convex-minimization-free
\end{IEEEkeywords}

\section{Introduction}\Label{S1}
Arimoto-Blahut algorithm is a famous algorithm to solve the optimization problem in information theory \cite{Arimoto,Blahut}. Originally, it aims the calculation of the channel capacity, i.e., the maximization of the mutual information. 
Later, it was extended to the calculation of the capacity of classical-quantum channel \cite{Nagaoka}. 
Recently, this algorithm was extended to a general minimization problem defined over the set of quantum states \cite{RISB}.
The paper \cite{HSF1} showed that
the iteration in the extended Arimoto-Blahut algorithm 
is the same as the iteration in the mirror descent algorithm among the above setting 
when the objective function is convex.
The extended Arimoto-Blahut algorithm
has the following advantage over the mirror descent algorithm. 
The extended Arimoto-Blahut algorithm
gives each iteration without any optimization in the above setting
while the mirror descent algorithm 
requires solving a convex minimization in each iteration.

Moreover, the extended Arimoto-Blahut algorithm
was extended to a general minimization problem defined over 
the set of probability distributions with linear constraint \cite{Hmixture}, and also 
that over the set of quantum states with linear constraint \cite{HL}.
Such a set with linear constraints is called a mixture family.
In statistics and information theory,
another type of a subset of distributions, 
an exponential family, takes an important role \cite{Amari-Nagaoka}. 
In information theory and machine learning,
people often focus on the miminum divergence problem between a given mixture family and a given exponential family.
This problem appears in Boltzmann machine \cite{AKN}.
The em-algorithm is known as a typical method to solve this problem \cite{H23}.
Further, the paper \cite{Hmixture} showed that the extended 
Arimoto-Blahut algorithm covers 
the minimization problem to be solved by the em-algorithm.
However, since the extended 
Arimoto-Blahut algorithm is still limited to functions defined over 
the set of probability distributions or the set of quantum states,
its applicable area is quite limited.
That is, the algorithm cannot be applied an optimization problem that has no relation with probability distributions or quantum states.
To extract the merit of the extended Arimoto-Blahut algorithm
even in a general optimization,
it is needed to formulate the extended 
Arimoto-Blahut algorithm in a more general setting.

In addition to the above problem, the existing extended Arimoto-Blahut algorithm 
has the following two problems.
As the first problem, it is unclear whether the equivalence relation with 
the mirror descent algorithm holds even under linear constraint.
As the second problem, 
the extended Arimoto-Blahut algorithm given in \cite{Hmixture}
requires the calculation of e-projection in each iteration.
An e-projection is a projection to a mixture family 
along an exponential family, and needs to solve a convex minimization whose number of variables equals the number of linear constraints.
This minimization step can be considered as a bottleneck 
in the extended Arimoto-Blahut algorithm of this case.

This problem is linked to the difficulty of the em-algorithm.
As presented in \cite{Hmixture},
the em-algorithm is a special case of 
the extended Arimoto-Blahut algorithm.
The em-algorithm is an algorithm to calculate 
the minimum divergence between a mixture family and an exponential family, and 
has been studied 
in the areas of machine learning and neural networks \cite{Amari1,Amari,Amari-Nagaoka,Allassonniere,Fujimoto}.
The em-algorithm is composed of the $e$-projection to the mixture family 
and the $m$-projection to the exponential family.
While the $m$-projection is given as an affine operation for the probability distribution,
the $e$-projection requires a more complicated calculation, i.e., a convex optimization.
Therefore, if the above 
bottleneck in the extended Arimoto-Blahut algorithm 
is resolved,
this method can be applied to the minimization of 
the divergence between a mixture family and an exponential family.

In fact, the em-algorithm is important even from the viewpoint of information theory as follows.
Originally, Blahut \cite{Blahut} studied 
the minimization of the mutual information 
in the context of rate-distortion theory
while rate-distortion theory can be applied to 
machine learning \cite{Deciding}.
Rate-distortion theory is formulated 
as an optimization problem of a joint distribution over 
given two system spaces with linear constraints.
That is, under the linear constraints,
we minimize the mutual information between these two systems.
Blahut \cite{Blahut}'s proposed algorithm minimizes only the sum of the mutual information and a constant times of the linear constraint, where the constant can be considered as 
the Lagrange multiplier.
He showed that there exists a constant such that 
the solution satisfies the given constraint,
but he did not present how to find such a good constant.
To resolve this problem, the recent paper \cite{H23}
found that the minimization of rate-distortion theory
can be solved by 
the em-algorithm.
When the em-algorithm is applied to the minimization,
the process of the e-projection
essentially seeks the suitable Lagrange multiplier.
That is, in the method \cite{H23},
each iteration updates the Lagrange multiplier.
Since the e-projection requires solving a convex minimization,
avoiding such a convex minimization
is essential even for the minimization of mutual information in rate-distortion theory.

This paper addresses the above three problems as follows.
First, we formulate
the extended Arimoto-Blahut algorithm by using 
Bregman-divergence.
This formulation allows us to handle a minimization problem
under a very general setting.
In this formulation, each iteration is given by using e-projection to a mixture family

Second, we show that 
the iteration in our extended Arimoto-Blahut algorithm 
is the same as the iteration in the mirror descent algorithm among the above setting  when the objective function is convex.
Although 
our extended Arimoto-Blahut algorithm and
the mirror descent algorithm
have the same iteration under the convexity condition, our extended Arimoto-Blahut algorithm
has the following advantage.
The mirror descent algorithm requires a convex minimization whose 
number of variables equals the number of original 
linear constraints.
When we choose a Bregman divergence in a suitable way,
we can avoid a convex minimization in each iteration.

Third, using the above type of choice of a Bregman divergence,
we propose a minimization-free-iteration
iterative minimization algorithm for the same problem studied in the paper \cite{Hmixture}.
This method can be applied to the minimization of 
the divergence between a mixture family and an exponential family, which includes the derivation of
the optimal conditional distribution for the rate-distortion theory.
This method iteratively modifies the objective function's input, potentially moving it outside the original domain.  Therefore, the objective function's domain must be extended.

The remainder of this paper is organized as follows.
Section \ref{S2} explains a Bregman divergence system as our preparation.
Section \ref{S3} formulates our minimization problem over a convex set
with Bregman divergence,
and presents our algorithm.
The presented general problem covers various problems
including 
channel coding \cite{RISB},  
Boltzmann machine \cite{AKN}, rate-distribution
theory on classical and quantum systems \cite{H23}.
Section \ref{S4} shows that the iteration in our is the same as the iteration in the mirror descent algorithm among the above setting 
when the objective function is convex.
Section \ref{S5} applies our algorithm to the case when the objective function is given over a set of probability distributions with linear constraints.
Then, we propose a minimization-free-iterative algorithm for this case.
Section \ref{S6} applies the algorithm given in 
Section \ref{S5} to the problem of the em-algorithm.
Section \ref{S7} applies it to the rate-distortion theory and makes a numerical analysis, where 
the minimum compression rate in the rate-distortion theory is 
given as the minimization of the mutual information
by changing the conditional distribution with fixed marginal distribution of the input system.
Section \ref{S8} applies our algorithm to the case of quantum states with linear constraints, 
which covers 
including classical-quantum channel coding \cite{RISB},
information bottleneck \cite{HY,HL},
quantum em algorithm, and
quantum rate-distribution theory \cite{H23}.
Sections \ref{S9} and \ref{S10}
are devoted to the proofs of theorems given in Section \ref{S3}.
Section \ref{S11} makes conclusions.

\section{Bregman divergence system}\Label{S2}
\subsection{Legendre transform}\Label{S2-0}
In this paper, 
a sequence $a= (a^i)_{i=1}^\ell$ with an upper index expresses
a vertical vector 
and 
a sequence $b= (b_i)_{i=1}^\ell$ with a lower index expresses
a horizontal vector as
\begin{align}
a= \left(
\begin{array}{c}
a^1 \\
a^2 \\
\vdots \\
a^\ell
\end{array}
\right), \quad
b= (b_1, b_2,\ldots, b_\ell).
\end{align}

We choose an open convex $\Theta$ set in $\mathbb{R}^{d}$ and 
a $C^2$-class strictly convex function $\phi:\Theta \rightarrow \mathbb{R}$.
Using the convex function $\phi$, we introduce another parametrization 
$\eta =(\eta_1, \ldots, \eta_{d})\in \mathbb{R}^{d}$ as
\begin{align}
\eta_j := \partial_j \phi(\theta),\Label{du1}
\end{align}
where $\partial_j$ expresses the partial derivative for the $j$-th variable
$\frac{\partial }{\partial \theta^j}$.
We also use the notation for the vector 
$\nabla^{(e)} [\phi](\theta):=
(\partial_j \phi(\theta))_{j=1}^{d}$.
Hence, the relation \eqref{du1} is rewritten as
\begin{align}
\eta (\theta)= \nabla^{(e)} [\phi](\theta).\Label{M1}
\end{align}
Therefore, $\nabla^{(e)}$ can be considered as a horizontal vector.

Since $\phi$ is $C^2$-class strictly convex function,
the map $\theta \to \eta (\theta)$ is one-to-one.
The parametrization 
$\eta_j$ is called the mixture parameter
while the parameter $\theta=(\theta^j)_j$ is called the natural parameter.
In the following, $\Xi$ expresses the open set of vectors $\eta(\theta) =(\eta_1, \ldots, \eta_{d})$ 
given in \eqref{du1}.
That is, $\nabla^{(e)} [\phi](\Theta)=\Xi$.
Hence, we denote the inverse function by $\eta\mapsto \theta(\eta)$
with the domain $\Xi$.
For $\eta \in \Xi$,
we define the {\it Legendre transform} 
$\phi^*$ of $\phi$ 
\begin{align}
\phi^*(\eta)=\sup _{\theta \in \Theta} \langle \eta,\theta\rangle -\phi(\theta).
\Label{MN1}
\end{align}

We denote the partial derivative for the $j$-th variable under the mixture parameter by $\partial^j$, i.e., $\frac{\partial }{\partial \eta^j} $.
The partial derivative of $\phi^*$ is given as
\cite[Section 3]{Fujimoto}\cite[Section 2.2]{hayashi}
\begin{align}
\partial^j\phi^*
(\eta(\theta) )=\theta^j. \Label{du2}
\end{align}
In the same way as the above, we use the notation
$\nabla^{(m)} [\phi^*](\eta):=(\frac{\partial \phi^*}{\partial \eta_j}(\eta))_{j=1}^{d}$.
The relation \eqref{du2} is rewritten as
\begin{align}
\theta = \nabla^{(m)} [\phi^*] (\eta(\theta) ).\Label{M2}
\end{align}
Therefore, it is also possible to start the parameter $\eta$ and the convex function 
$\phi^*$ and reproduce the parameter $\theta$ in the above way.

Next, we introduce the concept of  Bregman divergence, which is a generalization of the conventional divergence.
\begin{definition}[Bregman divergence]
We choose an open set $\Theta$ in $\mathbb{R}^{d}$ and 
a $C^2$-class strictly convex function $\phi:\Theta \rightarrow \mathbb{R}$.
We define the Bregman divergence $D^\phi$ as
\begin{align}
    D^{\phi}(\theta_1 \| \theta_2)
    :=& 
    \langle \nabla^{(e)}[\phi](\theta_1), \theta_1 - \theta_2\rangle 
    - \phi(\theta_1)+\phi(\theta_2)     \nonumber \\
=&\sum_{j=1}^{d} \eta_j(\theta_1) (\theta_1^j - \theta_2^j)
    - \phi(\theta_1)+\phi(\theta_2)   
\Label{XZL}
\end{align}
for $\theta_1, \theta_2 \in \Theta$.
\end{definition}
The triplet
$(\theta,\eta,D^{\phi}(\cdot \| \cdot))$
is called 
the Bregman divergence system defined by $\phi$.
When we use the parameter $\eta$, 
the Bregman divergence $D^\phi$ is rewritten as
\begin{align}
&    D^{\phi}(\theta(\eta) \| \theta(\eta'))
    =D^{\phi^*}(\eta' \| \eta)\\
    =& 
    \sum_{j=1}^{d}\theta(\eta')^j ( \eta'_j-\eta_j)+\phi^*(\theta(\eta) )
-\phi^*(\theta(\eta') ).
\Label{XZL3}
\end{align}

\begin{table}[t]
\caption{Notations with natural parameter}
\label{notation-1}
\begin{center}
\begin{tabular}{|c|l|c|}
\hline
Symbol &Description &Eq. number    \\
\hline
$\theta$ &natural parameter $(\theta_1,\ldots,\theta_d)$&    \\
\hline
$\Theta\subset \mathbb{R}^d$ & Parameter space for natural parameter&   \\
\hline
$\phi$ & convex function &   \\
\hline
$D^{\phi}$ & Bregman divergence for $\phi$& \eqref{XZL}  \\
\hline
$\partial^{j}$ & partial derivative with respect to $\theta_j$ &   \\
\hline
\multirow{2}{*}{$\nabla^{(e)}$} & 
vector composed of partial derivatives &   \\
&with respect to natural parameter &   \\
\hline
\end{tabular}
\end{center}
\end{table}

\begin{table}[t]
\caption{Notations with mixture parameter}
\label{notation-2}
\begin{center}
\begin{tabular}{|c|l|c|}
\hline
Symbol &Description &Eq. number    \\
\hline
$\eta$ &mixture parameter $(\eta^1,\ldots,\eta^d)$&    \eqref{du1}\\
\hline
$\eta_{(d_0)}$ & $(\eta_1,\ldots,\eta^{d_0})$ &   \\
\hline
$\Xi\subset \mathbb{R}^d$ & Parameter space for mixture parameter&   \\
\hline
$\phi^*$ & Legendre transform of convex function $\phi$ &  \eqref{MN1} \\
\hline
$D^{\phi^*}$ & Bregman divergence for $\phi^*$&  
\eqref{XZL3}
 \\
\hline
$\partial_{j}$ & partial derivative with respect to $\eta^j$ &   \\
\hline
\multirow{2}{*}{$\nabla^{(m)}$} & 
vector composed of partial derivatives &   \\
&with respect to mixture parameter &   \\
\hline
\end{tabular}
\end{center}
\end{table}

\subsection{Mixture family}\Label{S2-2}
Next, we introduce a mixture family, and discuss its properties.
For $d$ linearly independent vectors $u_1, \ldots, u_{d} \in \mathbb{R}^{d}$, and 
a vector $c=(c_1, \ldots, c_{k} )^T \in \mathbb{R}^{k}$, we say that
a subset $\mathcal{M} \subset \Theta$ is a {\it mixture family} 
generated by the constraint 
\begin{align}
\sum_{i=1}^{d} u^i_{d_0+j} 
\partial_i \phi(\theta) =c_j\Label{const1}
\end{align}
 for $j=1, \ldots, k$ and $d_0=d-k$
when the subset $\mathcal{M}$ 
is written as
\begin{align}
    \mathcal{M} = \left\{ \theta  \in \Theta \left| \hbox{ Condition \eqref{const1} holds.} \right. \right\} 
\end{align}
The $d \times d$ matrix $U$ is defined as $(u_1 \ldots u_{d} )$.
To make a parametrization in the above mixture family ${\cal M}$,
we set the new natural parameter $\bar{\theta}=(\bar{\theta}^1, \ldots, \bar{\theta}^{d})$ as
$\theta=U \bar{\theta} $,
and introduce the new mixture parameter
\begin{align}
\bar{\eta}_i
=\partial_i (\phi \circ U) (\bar{\theta}).
\Label{Dif}
\end{align}
Since the relation $\bar{\eta}_{d_0+i}= c_i$ holds for 
$i=1, \ldots, k$ 
in ${\cal M}$,
the initial $d_0$ elements $\bar{\eta}_1, \ldots, \bar{\eta}_{d_0}$
give a parametrization for ${\cal M}$.

Therefore, in the following, without loss of generality, 
replacing the parameterization of the natural parameter $\theta$ 
by $U^{-1}\theta$,
we assume that the mixture family is defined by the following constraint:
\begin{align}
\partial^{d_0+j}\phi (\theta) =c_j\Label{const1B}
\end{align}
 for $j=1, \ldots,k$.
To make the parametrization, 
we define the map 
$\psi_{\cal M}^{(m)}$ on ${\cal M}$
as 
$\psi_{\cal M}^{(m)} ( {\theta}):= (\partial_j \phi ({\theta}))_{j=1}^{d_0}$.
The set $\Xi_{{\cal M}}:= 
\{ \psi_{\cal M}^{(m)} (\theta) |  {\theta} \in {\cal M}\}$
works as the range of the new mixture parameters,
and we also employ the inverse map 
$(\psi_{{\cal M}}^{(m)})^{-1}: 
\Xi_{{\cal M}}\to {\cal M}$.

Next, we discuss how the mixture family ${\cal M}$ is characterized 
only by the parameters
$\theta_{(d_0)}:= (\theta^1, \ldots, \theta^{d_0}) \in \mathbb{R}^{d_0}$
and
$\eta^{(d_0)}:= (\eta^1, \ldots, \eta^{d_0}) 
\in \mathbb{R}^{d_0}$.
Then, we notice that
\begin{align}
\nabla^{(e)} [\phi]({\cal M})
=\{(\eta^{(d_0)},c_1,\ldots,c_k)\}_{\eta^{(d_0)}\in
\Xi_{{\cal M}}}.
\Label{CGU}
\end{align}
When an element ${\eta} \in \Xi_{{\cal M}}$ satisfies 
${\eta}_j =\partial_j \phi ({\theta})$ for $j=1, \ldots, d_0$,
we have
\begin{align}
\partial^i \phi^*( {\eta}^{(d_0)},c) = {\theta}^i\Label{CO1}
\end{align}
 for $i=1, \ldots, d_0$.
The strict convexity of the map $\phi^*_{{\cal M}}:
{\eta}^{(d_0)} \mapsto \phi^*( {\eta}^{(d_0)},c) $ guarantees that
the map ${\eta} \mapsto (\partial^i \phi^*( {\eta},c))_{i=1}^{d_0}$
is one-to-one. 
Hence, the initial $d_0$ elements $\theta_{(d_0)}=({\theta}^1, \ldots, {\theta}^{d_0})$
form a parametrization for ${\cal M}$.
In other words, the relation 
\begin{align}
( \theta^i )_{i=1}^{d_0}
= (\partial^i \phi^*( \psi_{\cal M}^{(m)} ( {\theta}) ,c))_{i=1}^{d_0} \Label{NAY}
\end{align}
holds.
We define the set $\Theta_{\cal M}:= 
\{ ( \theta^i )_{i=1}^{d_0} | {\theta} \in {\cal M} \}$.
By using the notation 
$\theta_{(d_0+1,d)}:= (\theta^{d_0+1}, \ldots, \theta^d) \in \mathbb{R}^{k}$,
the set $\Theta_{\cal M}$ is rewritten as
\begin{align}
&\Theta_{\cal M}\nonumber \\
=& 
\left\{\theta_{(d_0)} \in \mathbb{R}^{d_0} \left|
\begin{array}{l}
\exists \theta_{(d_0+1,d)}(\theta_{(d_0)}) \in \mathbb{R}^{k} 
\hbox{ such that}  \\
\partial_j \phi (\theta_{(d_0)}, \theta_{(d_0+1,d)}(\theta_{(d_0)})) =c_j \\
\hbox{ for } j=d_0+1, \ldots,d.
\end{array}
\right.\right\}.
\Label{const1-U}
\end{align}

We define the Legendre transform $\phi_{{\cal M}} $ of 
$\phi^*_{{\cal M}} $ as
\begin{align}
&\phi_{{\cal M}} (\theta_{(d_0)}):=
\sup_{\eta} \langle \eta^{(d_0)}, \theta_{(d_0)}\rangle 
- \phi^*_{{\cal M}} (\eta^{(d_0)})\nonumber \\
=&
\inf_{\theta^{d_0+1}, \ldots, \theta^{d}} 
\phi (\theta^{(d_0)},\theta^{d_0+1}, \ldots, \theta^{d})
-\sum_{j=1}^{k} \theta^{d_0+j} c_j.
\end{align}
Then, we have
\begin{align}
&D^\phi((\theta_{(d_0)}, \theta_{(d_0+1,d)}(\theta_{(d_0)}))\|
(\theta_{(d_0)}', \theta_{(d_0+1,d)}(\theta_{(d_0)}'))) \nonumber \\
=
&D^{\phi^*} \Big(
(\partial^{j}\phi (\theta_{(d_0)}, \theta_{(d_0+1,d)}(\theta_{(d_0)})))_{j=1}^{d}
\Big\|\nonumber \\
&\quad
\partial^{j}\phi(\theta_{(d_0)}', \theta_{(d_0+1,d)}(\theta_{(d_0)}')))_{j=1}^{d}\Big)
 \nonumber \\
=
&D^{\phi^*_{\cal M}}\Big(
(\partial^{j}\phi (\theta_{(d_0)}, \theta_{(d_0+1,d)}(\theta_{(d_0)}))
)_{j=1}^{d_0}
\Big\|\nonumber \\
&\quad
(\partial^{j}\phi(\theta_{(d_0)}', \theta_{(d_0+1,d)}(\theta_{(d_0)}'))
)_{j=1}^{d_0}
\Big) \nonumber \\
=&
D^{\phi_{{\cal M}}}(\theta_{(d_0)}\|\theta_{(d_0)}').
\end{align}
Therefore, 
the mixture family ${\cal M}$
can be characterized by 
the Bregman divergence system defined by $\phi_{\cal M}$.

We define the $e$-projection 
$\Pro^{(e),\phi}_{\mathcal{M}}$
to ${\cal M}$ as \cite{Amari1,Amari,Amari-Nagaoka}
\cite[Eq. (53)]{H23}\footnote{The reference \cite{H23} uses the terminology
$e$-projection and $m$-projection in the opposite way.
Since the projection to a mixture family ${\cal M}$ is done along an exponential family,
it should be called 
the $e$-projection to a mixture family ${\cal M}$.}
\begin{align}
\Pro^{(e),\phi}_{\mathcal{M}} (\overline{\theta}):=
\argmin_{\theta' \in {\cal M}}
D^\phi(\theta'\|\overline{\theta}).\Label{ZCV}
\end{align}
For an element $\theta \in \mathcal{M}$
and a general element $\overline{\theta} \in \Theta$,
the $e$-projection $\Pro^{(e),\phi}_{\mathcal{M}}$
satisfies 
Pythagorean Theorem for Bregman divergences 
\cite{Amari-Nagaoka},\cite[Proposition 1 and Lemma 2]{H23} as
\begin{align}
D^\phi(\theta\|\overline{\theta})
=
D^\phi(\theta\| \Pro^{(e),\phi}_{\mathcal{M}} (\overline{\theta}))
+D^\phi(\Pro^{(e),\phi}_{\mathcal{M}} (\overline{\theta})\| \overline{\theta}).
\Label{BVY}
\end{align}
This relation is a key equation in information geometry.
The calculation method for the $e$-projection 
$\Pro^{(e),\phi}_{\mathcal{M}} (\overline{\theta})$
is explained in 
\cite{H23} by solving a convex minimization as follows.
To explain its detail, we need to explain the exponential family 
\begin{align}
{\cal E}:=
\{(\overline{\theta}_{d_0},\theta_{(d_0+1,d)})
| \theta_{(d_0+1,d)} \in \mathbb{R}^{k}\}
\end{align}
that contains $\overline{\theta}$.
The $e$-projected element $\Pro^{(e),\phi}_{\mathcal{M}} (\overline{\theta})$
belongs to the mixture family ${\cal M}$
and the exponential family ${\cal E}$.
Hence, the $e$-projected element $\Pro^{(e),\phi}_{\mathcal{M}} (\overline{\theta})$ has the form 
$(\overline{\theta}_{d_0},\theta_{(d_0+1,d)})$.
That is, we need to identify $\theta_{(d_0+1,d)}$.
Due to \eqref{const1B},
the condition $(\overline{\theta}_{d_0},\theta_{(d_0+1,d)}) \in {\cal M}$
is equivalent to 
\begin{align}
\partial_{d_0+i}\phi(\overline{\theta}_{d_0},\theta_{(d_0+1,d)})=c_i \Label{SBN}
\end{align}
for $i=1, \ldots, k$.
Since $\phi$ is convex function, the solution of \eqref{SBN} is the minimizer of 
$\min_{\theta_{(d_0+1,d)}}
\phi(\overline{\theta}_{d_0},\theta_{(d_0+1,d)})$.

\begin{table}[t]
\caption{Notations related to mixture family}
\label{notation-3}
\begin{center}
\begin{tabular}{|c|l|c|}
\hline
Symbol &Description &Eq. number    \\
\hline
${\cal M}\subset \mathbb{R}^{d}$ &mixture family &    \eqref{du1}\\
\hline
$\Theta_{\cal M}\subset \mathbb{R}^{d_0}$ &natural parameter for ${\cal M}$ 
& \eqref{du1}\\
\hline
$\Xi_{\cal M}\subset \mathbb{R}^{d_0}$ &mixture parameter for ${\cal M}$&    \eqref{du1}\\
\hline
$\theta_{(d_0)}$ & $(\theta^1,\ldots,\theta^{d_0})$ &   \\
\hline
$\theta_{(d_0+1,d)}$ & 
$(\theta^{d_0+1},\ldots,\theta^{d})$ &   \\
\hline
$\eta^{(d_0)}$ & $(\eta_1,\ldots,\eta_{d_0})$ &   \\
\hline
$\Pro^{(e),\phi}_{\mathcal{M}}$ & 
$e$-projection to $\mathcal{M}$ & \eqref{ZCV}  \\
\hline
$k$ & 
number of linear constraints, $d-d_0$&  \\
\hline
\end{tabular}
\end{center}
\end{table}

\section{Bregman-divergence-based Arimoto-Blahut algorithm}\Label{S3}
\subsection{Our general algorithm}
In this paper, we address 
the minimization with the following objective function
\begin{align}
\tilde{\cal G}(\eta):=\sum_{j=1}^d \eta_j \tilde{\Omega}^j(\eta)
\Label{form2}
\end{align}
with a function $\tilde{\Omega}$ from a convex subset 
${\cal D} \subset \mathbb{R}^d$ to $\mathbb{R}^d$.
That is, our problem is formulated as the following two problems;
\begin{align}
{\cal T}:=\min_{\eta \in {\cal D}} \tilde{\cal G}(\eta), \quad
\eta_{*}:=\argmin_{\eta \in {\cal D}} \tilde{\cal G}(\eta).
\label{RBE2}
\end{align}

To address the above problem, 
we assume that there exist
a convex function $\phi$ defined an open subset 
$\Theta \subset \mathbb{R}^d$
and
a mixture family ${\cal M}$ of 
the Bregman divergence system $(\theta,\eta,D^{\phi}(\cdot \| \cdot))$
defined by $\phi$
such that the convex subset 
${\cal D} \subset \mathbb{R}^d$ equals
the set $\Xi_{{\cal M}}$ of the mixture parameters 
of the mixture family ${\cal M}$.
Using 
the one-to-one map $\eta \to \theta (\eta)$ 
defined by the partial derivative of $\phi$,
we consider the above problems with the coordinate $\theta$. 
For this aim, we define the function as
\begin{align}
{\cal G}(\theta):=\sum_{j=1}^d \eta_j(\theta) \Omega^j(\theta),
\quad
\Omega(\theta):=\tilde{\Omega}(\eta(\theta)).
\Label{form1}
\end{align}
The above minimization is rewritten as 
\begin{align}
{\cal T}=\min_{\theta \in {\cal M}} {\cal G}(\theta), \quad
\theta_{*}:=\argmin_{\theta \in {\cal M}} {\cal G}(\theta).
\label{RBE}
\end{align}
The following discussion is based on the form \eqref{form1}.

We define the conversion function ${\cal F}_\gamma$ from $\Theta$ to $\Theta$
as
\begin{align}{\cal F}_\gamma (\theta):= 
\theta -\frac{1}{\gamma} \Omega(\theta).
\Label{BF6}
\end{align}
Then, we propose Algorithm \ref{AL1}.
When the calculation of $ \Omega(\theta)$ and 
the $e$-projection is feasible, Algorithm \ref{AL1} is feasible.

\begin{algorithm}
\caption{BD-based AB algorithm for ${\cal G}(\theta)$}
\Label{AL1}
\begin{algorithmic}
\STATE {Choose the initial value $\theta^{[1]} \in {\cal M}$;} 
\REPEAT 
\STATE Calculate $\theta^{[t+1]}:=
\Pro^{(e),\phi}_{\mathcal{M}} \circ {\cal F}_\gamma(\theta^{[t]})$;
\UNTIL{convergence.} 
\end{algorithmic}
\end{algorithm}

Then, the following two theorems hold for Algorithm \ref{AL1}.

\begin{theorem}\Label{NLT}
When all pairs $(\theta^{[t]},\theta^{[t+1]})$ satisfy 
the following condition with 
$(\theta,\theta')=(\theta^{[t]},\theta^{[t+1]})$
\begin{align}
D_\Omega(\theta\|\theta'):=\sum_{j=1}^d \eta_j(\theta) 
(\Omega^j(\theta)- \Omega^j(\theta'))
&\le \gamma D^\phi(\theta\|\theta')
\Label{BK1+} ,
\end{align}
for some sufficiently large positive number $\gamma$,
Algorithm \ref{AL1} 
always iteratively improves the value of the objective function. 
\end{theorem}

The condition \eqref{BK1+} is rewritten by using the mixture parameter as
\begin{align}
\tilde{D}_{\tilde{\Omega}}(\eta\|\eta'):=\sum_{j=1}^d \eta_j
(\tilde{\Omega}^j(\eta)- \tilde{\Omega}^j(\eta'))
&\le \gamma D^{\phi^*}(\eta'\|\eta)
\Label{BK1+M} .
\end{align}

As a generalization of \cite[Theorem 3.3]{RISB},  
the following theorem discusses the convergence to the global minimum 
and the convergence speed.
\begin{theorem}\Label{TH1}
When any two densities $\theta$ and $\theta'$ in $\Theta$ 
satisfy the condition \eqref{BK1+}, and
the element \red{$\theta=\theta_* $} satisfies 
\begin{align}
D_\Omega (\theta\|\theta') \ge 0 
\Label{BK2+} 
\end{align}
with any element $\theta' $,
Algorithm \ref{AL1} satisfies the condition
\begin{align}
{\cal G}(\theta^{[t_0+1]})
-{\cal G}(\theta_{*})
\le 
\frac{\gamma D^\phi(\theta_{*}\| \theta^{[1]}) }{t_0} \Label{XME}
\end{align}
with any initial element $\theta^{[1]}$.
\end{theorem}

As explained in Section \ref{S5},
when the Bregman divergence $D^{\phi}(\theta\|\theta')$
is given as KL divergence,
our algorithm (Algorithm \ref{AL1}) coincides with the algorithm presented in \cite{Hmixture}.
Also, 
as explained in Section \ref{S8},
when the Bregman divergence $D^{\phi}(\theta\|\theta')$
is given as quantum relative entropy,
our algorithm (Algorithm \ref{AL1}) coincides with the algorithm presented in \cite{HL}.

\begin{table}[t]
\caption{Notations related to objective function}
\label{notation-4}
\begin{center}
\begin{tabular}{|c|l|c|}
\hline
Symbol &Description &Eq. number    \\
\hline
\multirow{2}{*}{$\tilde{\cal G}$} &objective function as  & \multirow{2}{*}{Eq. \eqref{form2}}     \\
&a function of mixture parameter& \\
\hline
\multirow{2}{*}{$\tilde{\Omega}^j$} & function of  & \multirow{2}{*}{Eq. \eqref{form2}}     \\
&mixture parameter& \\
\hline
\multirow{2}{*}{${\cal G}$} &objective function as  & \multirow{2}{*}{Eq. \eqref{form1}}     \\
&a function of mixture parameter& \\
\hline
\multirow{2}{*}{${\Omega}^j$} & function of  & \multirow{2}{*}{Eq. \eqref{form1}}     \\
&natural parameter& \\
\hline
${\cal F}_\gamma $& 
conversion over natural parameter&
\eqref{BF6} \\
\hline
$D_\Omega(\theta\|\theta')$ &
two-input function of natural parameter
& \eqref{BK1+} \\
\hline
$\tilde{D}_{\tilde{\Omega}}(\eta\|\eta')$ &
two-input function of mixture parameter
&\eqref{BK1+M} 
\\
\hline
$\theta_{*}$ & minimizer of $G$ & \eqref{RBE}
\\
\hline
\end{tabular}
\end{center}
\end{table}

\subsection{Calculation of iteration process}
Here, we discuss how to execute the iteration process
$\theta^{[t+1]}:=
\Pro^{(e),\phi}_{\mathcal{M}} \circ {\cal F}_\gamma(\theta^{[t]})$
when the mixture family ${\cal M}$ is characterized by \eqref{const1B}.
In this case, 
${\cal F}_\gamma(\theta^{[t]})$ satisfies 
\begin{align}
{\cal F}_\gamma(\theta^{[t]})^j= (\theta^{[t]})^j-\frac{1}{\gamma}
\Omega^j(\theta(\eta_{[t]}))
\Label{SG1}
\end{align}
for $j=1, \ldots, d$.
Since $\Pro^{(e),\phi}_{\mathcal{M}}$ is an e-projection,
it does not change the initial $d_0$ parameters $\theta^1,\ldots,\theta^{d_0}$.
Hence, we have
\begin{align}
(\Pro^{(e),\phi}_{\mathcal{M}} \circ{\cal F}_\gamma(\theta^{[t]}))^j
= (\theta^{[t]})^j-\frac{1}{\gamma} \Omega^j(\theta(\eta_{[t]}))
\Label{SG2}
\end{align}
for $j=1, \ldots, d_0$.
However, the calculation of 
$(\Pro^{(e),\phi}_{\mathcal{M}} \circ{\cal F}_\gamma(\theta^{[t]}))^{d_0+i}$ with $i=1, \ldots, k(=d-d_0)$
is not so trivial, and 
these values are needed to calculate 
$\Omega^j(\theta(\eta_{[t+1]}))$
in the next iteration.
That is,
we need to determine
the $k$ parameters $\theta^{d_0+1}, \ldots, \theta^{d}$
by solving the differential equation
\eqref{const1B} with respect to 
the $k$ parameters $\theta^{d_0+1}, \ldots, \theta^{d}$
with given the $d_0$ parameters
$\theta^{1}, \ldots, \theta^{d_0}$
When $k$ is larger, the calculation process of the $e$-projection is the bottleneck of our algorithm.
However, this number $k$ can be reduced to $1$ even in the general case as follows.

Since our parametrization satisfies \eqref{const1B},
$\eta_1,\ldots, \eta_{d_0}$ are free parameters
and 
$\eta_{d_0+1},\ldots, \eta_{d}$ are fixed to constants.
In this case, 
the function $\tilde{\cal G}$ is written as
\begin{align}
\tilde{\cal G}(\eta)=
\sum_{j=1}^{d_0} \eta_j \tilde{\Omega}^j(\eta)
+\sum_{j=d_0+1}^d \eta_j \tilde{\Omega}^j(\eta)
\Label{form3}.
\end{align}
When we define the new function 
$\overline{\Omega}^{d_0+1}(\eta):=
\sum_{j=d_0+1}^d \eta_j \tilde{\Omega}^j(\eta)$, 
the function $\tilde{\cal G}$ is rewritten as
\begin{align}
\tilde{\cal G}(\eta)=
\sum_{j=1}^{d_0} \eta_j \tilde{\Omega}^j(\eta)
+\overline{\Omega}^{d_0+1}(\eta)
\Label{form4}.
\end{align}
That is, our problem setting can be reduced to the case when $d=d_0+1$ and 
$\eta_{d_0+1}$ is constrained to be $1$.

In this case, there are several good choices for 
the function $\phi$ such that
the $e$-projection 
$\Pro^{(e),\phi}_{\mathcal{M}} (\theta) $
to ${\cal M}:= \{\eta \in \mathbb{R}^{d_0+1}|\eta_{d_0+1}=1\}$
can be easily calculated as follows.

\begin{example}\Label{Ex1}
As the first example, 
we choose where functions $f_1, \ldots,f_{d_0}$ on $\cX$ such that
$f_1, \ldots,f_{d_0}$ and the constant function
are linearly independent. 
For $\theta_{(d_0+1)} \in \mathbb{R}^{d_0+1}$, 
we define the convex function
\begin{align}
\phi_\kappa(\theta_{(d_0+1)}):=
\sum_{x \in \cX} e^{\sum_{j=1}^{d_0}f_j(x) \theta^j + \theta^{d_0+1}}.
\Label{BNX}
\end{align}
This example will be used in Section \ref{S5-B}.
Then, taking the partial derivative for $\theta^j$,
we have
\begin{align}
\eta_j(\theta)=&\frac{\partial \phi_\kappa}{\partial \theta ^{j}}(\theta)
=\sum_{x \in \cX} f_j(x) e^{\sum_{j'=1}^{d_0}f_{j'}(x) \theta^{j'} + \theta^{d_0+1}} \Label{B6K}\\
\eta_{d_0+1}(\theta)=&\frac{\partial \phi_\kappa}{\partial \theta ^{d_0+1}}(\theta)
=\sum_{x \in \cX} e^{\sum_{j'=1}^{d_0}f_{j'}(x) \theta^{j'} + \theta^{d_0+1}} \notag\\
&=\bigg(\sum_{x \in \cX} e^{\sum_{j'=1}^{d_0}f_{j'}(x) \theta^{j'}}\bigg)
e^{\theta^{d_0+1}} 
\end{align}
for $j=1, \ldots, d_0$.
Hence, 
the $e$-projection 
$\Pro^{(e),\phi_\kappa}_{\mathcal{M}} (\theta) $ 
is given as
\begin{align}
(\Pro^{(e),\phi_\kappa}_{\mathcal{M}} (\theta))^j &=\theta^j, \quad
(\Pro^{(e),\phi_\kappa}_{\mathcal{M}} (\theta))^{d_0+1} =
\log \phi_\kappa(\theta).
\end{align}
for $j=1, \ldots, d_0$.
\end{example}

\begin{example}
As the second example, we choose $\phi(\theta):=\frac{1}{2}
\sum_{x \in \cX} (\sum_{j=1}^{d_0}f_j(x) \theta^j + \theta^{d_0+1})^2$.
Then, we have
\begin{align}
\frac{\partial \phi}{\partial \theta ^{j}}(\theta)
&=\sum_{x \in \cX} f_j(x) \bigg(\sum_{j'=1}^{d_0}f_{j'}(x) \theta^{j'} + \theta^{d_0+1}\bigg) \\
\frac{\partial \phi}{\partial \theta ^{d_0+1}}(\theta)
&=\sum_{x \in \cX} \bigg(\sum_{j'=1}^{d_0}f_{j'}(x) \theta^{j'} + \theta^{d_0+1}\bigg) \notag\\
&=\bigg(\sum_{x \in \cX} \sum_{j'=1}^{d_0}f_{j'}(x) \theta^{j'}\bigg)
 +|\cX|  \theta^{d_0+1}
\end{align}
for $j=1, \ldots, d_0$.
Hence, 
the $e$-projection 
$\Pro^{(e),\phi}_{\mathcal{M}} (\theta) $ 
is given as
\begin{align}
(\Pro^{(e),\phi}_{\mathcal{M}} (\theta))^j &=\theta^j, \\
(\Pro^{(e),\phi}_{\mathcal{M}} (\theta))^{d_0+1} &=
-\frac{1}{|\cX|}\bigg(\sum_{x \in \cX} \sum_{j=1}^{d_0}f_j(x) \theta^j\bigg)
\end{align}
for $j=1, \ldots, d_0$.
\end{example}


\section{Comparison with mirror descent}\Label{S4}
\begin{lemma}
When the condition \eqref{BK2+} holds for any $\theta,\theta'$,
the function $\tilde{\cal G}(\eta)$ is convex.
\end{lemma}

\begin{proof}
We choose any elements $\eta,\eta'' \in 
\Xi_{\cal M}$.
Then, given $\lambda \in (0,1)$, we choose 
the element $\eta'=\lambda \eta+(1-\lambda)\eta'' \in \Xi_{\cal M}$.
the condition \eqref{BK2+} for any $\theta=\theta(\eta),\theta'=\theta(\eta')$ implies
 \begin{align}
&\tilde{\cal G}(\eta)
\ge \sum_{j=1}^d \eta_j \Omega^j(\theta(\eta')) \notag\\
=& \sum_{j=1}^d \eta_j' \Omega^j(\theta(\eta'))
 +\sum_{j=1}^d (\eta_j-\eta_j') \Omega^j(\theta(\eta'))\notag\\
=&\tilde{\cal G}(\eta')
 +\sum_{j=1}^{d} (\eta_j-\eta_j') \Omega^j(\theta(\eta')).
\end{align}
Therefore, we have
\begin{align}
\tilde{\cal G}(\eta) &\ge \tilde{\cal G}(\eta')
 +(1-\lambda)\sum_{j=1}^{d} (\eta_j-{\eta''}_j) \Omega^j(\theta(\eta'))\Label{ZNI}.
\end{align}
Replacing the role of $\eta$ by $\eta''$, we have
\begin{align}
\tilde{\cal G}(\eta'') &\ge \tilde{\cal G}(\eta')
 -\lambda\sum_{j=1}^{d} (\eta_j-{\eta''}_j) \Omega^j(\theta(\eta')).
\end{align}
Thus,
\begin{align}
\lambda \tilde{\cal G}(\eta)+(1-\lambda )\tilde{\cal G}({\eta''})
\ge \tilde{\cal G}(\eta').
\end{align}
\end{proof}

\begin{lemma}\Label{BVR}
When the condition \eqref{BK2+} holds for any $\theta,\theta'$,
we have
\begin{align}
\frac{\partial \tilde{\cal G}}{\partial \eta_j}(\eta)
=  \tilde{\Omega}^j(\eta)
=  \Omega^j(\theta(\eta)).\Label{ZXI}
\end{align}
\end{lemma}

To prove Lemma \ref{BVR}, we describe Lemma 1 of \cite{HSF1} in our notation.
For a convex function $\tilde{\cal G}: \Xi \subset 
\mathbb{R}^{d}
\to \mathbb{R}$,
the subdifferential of $\tilde{\cal G}$ at $\eta \in \Xi$ is defined as
\begin{align}
\partial \tilde{\cal G}(\eta):=
\{v \in \mathbb{R}^{d} | \tilde{\cal G}(\eta')\ge \tilde{\cal G}(\eta)
+\langle v,\eta-\eta'\rangle , \forall \eta' \in \Xi \}.
\end{align}
An element of the subdifferential $v \in  \partial \tilde{\cal G}(\eta)$ is called the subgradient of $\tilde{\cal G}$ at $\eta$.
In the following, we denote the interior of 
$\Xi$ by $\int (\Xi)$.

\begin{lemma}[\protect{\cite[Lemma 1]{HSF1}}]\Label{L4-11}
Consider a convex function 
$\tilde{\cal G}: \Xi \subset 
\mathbb{R}^{d} \to \mathbb{R}$.
If there exists a single-valued continuous
operator $\tilde{\Omega} : 
\int (\Xi) \to \mathbb{R}^{d}$
such that $\tilde{\Omega}(\eta)\in  \partial \tilde{\cal G}(\eta)$ 
for any element $\eta \in \int(\Xi) $,
 then $\tilde{\cal G}$ is differentiable
on $\int \Xi$ and $\nabla \tilde{\cal G}(\eta) = 
\tilde{\Omega}(\eta)$ for any $\eta \in \int (\Xi)$.
\end{lemma}

\begin{proofof}{Lemma \ref{BVR}}
The relation \eqref{ZNI} guarantees that 
$  (\tilde{\Omega}^j(\eta))_{j=1}^d =  (\Omega^j(\theta(\eta)))_{j=1}^d$
belongs to $\partial \tilde{\cal G}(\eta)$.
Hence, Lemma \ref{L4-11} implies \eqref{ZXI}.
\end{proofof}

Now, we assume that our parametrization satisfies \eqref{const1B}.
Then, we describe the mirror descent algorithm as Algorithm \ref{AL2} \cite[Section 4.2]{Bubeck}, \cite[Algorithm 1]{HSF1}, \cite[(3.11)]{BT}.
Algorithm \ref{AL2} 
employs only the mixture parameter.
The equation \eqref{CGU} shows the relation between 
$\nabla^{(e)} [\phi]({\cal M})$ and $\Xi_{{\cal M}}$.

\begin{algorithm}
\caption{mirror descent algorithm for $\tilde{\cal G}(\eta)$}
\Label{AL2}
\begin{algorithmic}
\STATE {Choose the initial value 
$\eta_{[1]} \in \nabla^{(e)} [\phi]({\cal M})
\subset \mathbb{R}^d$;} 
\REPEAT 
\STATE Calculate $\eta_{[t+1]}:=
\argmin_{\eta \in \Xi_{\cal M}}
\sum_{j=1}^{d_0} \eta_j \frac{\partial \tilde{\cal G}}{\partial \eta_j}(\eta_{[t]})
+\frac{1}{\beta} D^\phi( \theta(\eta) \| \theta(\eta_{[t]}))$;
\UNTIL{convergence.} 
\end{algorithmic}
\end{algorithm}

\begin{theorem}\Label{TH6}
When the condition \eqref{BK2+} holds for any $\theta,\theta'$,
our algorithm 
is the same as 
the mirror descent algorithm.
\end{theorem}

Therefore, our algorithm can be considered as
a special case when 
the condition \eqref{BK2+} holds
and $\tilde{\cal G}$ is differentiable.
However, when the conditions of Theorem \ref{TH1} hold,
our algorithm has the convergence to the global minimum
even when $\tilde{\cal G}$ is differentiable.
Hence, our algorithm can be considered as a non-differentiable extension of 
the mirror descent algorithm.

Further, even when the condition \eqref{BK2+} does not hold,
Theorem \ref{NLT} guarantees that 
our algorithm monotonically decreases the objective function.
In fact, as numerically demonstrated in \cite{HL},
our algorithm has a relatively good performance.
Therefore, our algorithm can be used for a wider situation than the mirror descent algorithm
as long as the objective function has the form
\eqref{form1}.

\begin{proofof}{Theorem \ref{TH6}}
Since our parametrization satisfies \eqref{const1B},
$\eta_1, \ldots, \eta_{d_0}$ are free parameters
and $\eta_{d_0+1}, \ldots , \eta_{d}$ are defined to the fixed values.
Then, in the mirror descent algorithm,
the function of $\eta \in \Xi_{\cal M}$
to be minimized at the determination of 
$\eta_{[t+1]}$ is calculated by using Lemma \ref{BVR} as
\begin{align}
&\sum_{j=1}^{d_0} \eta_j \frac{\partial \tilde{\cal G}}{\partial \eta_j}(\eta_{[t]})
+\frac{1}{\kappa} D^\phi( \theta(\eta) \| \theta(\eta_{[t]}))\notag\\
=&
\sum_{j=1}^{d_0} \eta_j \Omega^j(\theta(\eta_{[t]}))
+\frac{1}{\kappa} \Big(
\sum_{j=1}^{d_0} \theta^j(\eta_{[t]})(\eta_{[t],j}-\eta_j) \notag\\
&+\phi^*(\eta)-\phi^*(\eta_{[t]}) \Big).
\end{align}
Here, we use the fact $\eta_{d_0+i}= \eta_{[t],d_0+i}$
with $i=1, \ldots, k$ for $\eta \in \Xi_{\cal M}$.
The partial derivative of the above value with respect to $\eta_j$ 
with $j=1, \ldots, d_0$
is the following condition for $\eta \in \Xi_{\cal M}$:
\begin{align}
&\Omega^j(\theta(\eta_{[t]}))-\frac{1}{\kappa} \theta^j(\eta_{[t]})
+\frac{1}{\kappa} \frac{\partial \phi^*}{\partial \eta_j}(\eta) \notag\\
=&
\Omega^j(\theta(\eta_{[t]}))-\frac{1}{\kappa} \theta^j(\eta_{[t]})
+\frac{1}{\kappa} \theta^j(\eta) .
\end{align}
That is, in the mirror descent algorithm,
$\eta_{[t+1]}\in \Xi_{\cal M}$ is chosen to be the element to satisfy the following condition.
\begin{align}
 \theta^j(\eta_{[t+1]})=
 \theta^j(\eta_{[t]})-\beta\Omega^j(\theta(\eta_{[t]}))
\end{align}
for $j=1, \ldots, d_0$, which coincides with the condition 
\eqref{SG2}.
When $\beta$ is chosen as $\frac{1}{\gamma}$,
the mirror descent algorithm coincides with our algorithm.
\end{proofof}

\section{Application to mixture family of probability distributions}\Label{S5}
\subsection{Formulation}
We apply our method
to the case when ${\cal M}$ is given as a mixture family of probability distributions.
We consider a finite sample space
${\cal X}$ where the cardinality of $\cX$ is $d'$.
We introduce $d'-1$ functions $f_1,\ldots, f_{d'-1}$ over $\cX$
such that 
$f_{1}, \ldots, f_{d'-1}$ and the constant function are linearly independent.
We also choose 
$d'$ functions $g^1, \ldots, g^{d'}$ on $\cX$ such that
\begin{align}
\sum_{x \in \cX} f_{j}(x)g^i(x)=\delta^i_{j},
\end{align}
where $f_{d'}(x):=1$.
We define a mixture family ${\cal M}_p$
as the set of distributions to satisfy the linear constraints
\begin{align}
\sum_{x \in \cX}P_X(x)f_j(x)=c_j \hbox{ for }j=d_0+1, \ldots, d'-1.\Label{BNI}
\end{align}
Then, we have two kinds of parametrization of the distribution on $\cX$.
Using the natural parameter $\theta \in \mathbb{R}^{d'-1}$,
we parameterize the distribution as
$P_\theta(x):= 
e^{\sum_{j=1}^{d'-1} f_j(x) \theta^j-\theta^{d'}(\theta)}$,
where $\theta^{d'}(\theta):=-\log 
\big(\sum_{x'\in \cX} e^{\sum_{j=1}^{d'-1} f_j(x') \theta^j}\big)$.
Using the mixture parameter $\eta \in \mathbb{R}^{d'-1}$,
we parameterize the distribution as
$\tilde{P}_\eta(x):=\sum_{j=1}^{d'-1}
\eta_j g^j(x)+ g^{d'}(x)$.
The distribution $\tilde{P}_\eta$ belongs to 
${\cal M}_p$
if and only if
$\eta_{j}=c_j$ for $j=d_0+1, \ldots, d'-1$.

Then, we introduce a function $\Omega_p[P_X]$ on $\cX$, which depends on 
the distribution $P_X \in {\cal M}_p$.
We define an objective function
${\cal G}_p(P_X):=\sum_{x \in \cX}P_X(x) \Omega_p[P_X](x)$, and
consider the minimization problem
\begin{align}
\min_{P_X \in {\cal M}_p} {\cal G}_p(P_X).\Label{BNY}
\end{align}
Using the Kullback-Leibler divergence
$D(P_X\|Q_X):=\sum_{x \in {\cal X}} P_X(x)(\log P_X(x)-\log Q_X(x))
$,
the paper \cite{Hmixture} introduced an algorithm
under the condition
\begin{align}
\sum_{x \in \cX}P_X(x) (\Omega_p[P_X](x)-\Omega_p[Q_X](x))
\le \gamma D(P_X\|Q_X).
\end{align}
The previous algorithm \cite{Hmixture} coincides with our algorithm 
under the following choices.

We define the function $\phi(\theta,\theta^{d'})$ on $\mathbb{R}^{d'}$ as 
\begin{align}
\phi(\theta,\theta^{d'})=\sum_{x\in \cX} e^{\sum_{j=1}^{d'-1} f_j(x) \theta^j+\theta^{d'}} .
\Label{BCR}
\end{align}
We choose the mixture family ${\cal M}$ as
\begin{align}
{\cal M} 
:= \Bigg\{(\theta,\theta^{d'}) \in \mathbb{R}^{d'}
\Bigg| 
\begin{array}{l}
\frac{\partial \phi}{\partial \theta^j}(\theta,\theta^{d'})
=c_j \\
\frac{\partial \phi}{\partial \theta^{d}}(\theta,\theta^{d'})
=1 
\end{array}
\Bigg\},
\end{align}
where the index $j$ in the above condition
takes values $d_0+1, \ldots d'-1$.
The relation $P_\theta \in {\cal M}_p$ holds 
if and only if $(\theta,\theta^{d'}(\theta)) 
\in {\cal M}$.
In other words,
the relation $\tilde{P}_\eta \in {\cal M}_p$ holds 
if and only if $
(\eta,1) \in \Xi_{\cal M}$.
For 
$\eta,\eta'$ such that 
$(\eta,1),(\eta',1)\in \Xi_{\cal M} $,
we choose
\begin{align}
\tilde{\Omega}^j(\eta):=
\sum_{x\in \cX}g^j(x) \Omega_p[\tilde{P}_\eta](x)
\hbox{ for }j=1, \ldots, d
\end{align}
and have
\begin{align}
{\cal G}_p(\tilde{P}_\eta)
&= \tilde{\cal G}(\eta):=
\sum_{j=1}^{d'-1} \eta_j\tilde{\Omega}^j(\eta) 
+\tilde{\Omega}^d(\eta) \Label{O1}\\
&=\sum_{x \in \cX}\tilde{P}_\eta(x) 
\Omega_p[\tilde{P}_{\eta}](x),
\Label{O2}
\\
D(\tilde{P}_\eta\|\tilde{P}_{\eta'})&=
D^{\phi^*}((\eta',1)\|(\eta,1)).
\Label{O3}
\end{align}

Therefore, the algorithm in the previous paper \cite{Hmixture}
coincides with our algorithm with 
the objective function $ \tilde{\cal G}$ defined in \eqref{O1}
and the Bregman Divergence system 
based on the convex function $\phi$ defined by \eqref{BCR}.
However, the previous algorithm and our algorithm of the above choice 
have the process for the $e$-projection 
$\Pro^{(e),\phi}_{\mathcal{M}} (\theta) $.
Its step needs to solve a convex minimization with $d-d_0-1$
parameters. 

\subsection{Minimization-free-iteration algorithm}\Label{S5-B}
To avoid the convex minimization in each iteration, 
for $\eta^{(d_0)} \in \mathbb{R}^{d_0}$, 
using $c=(c_j)_{j=d_0+1}^{d'-1}$,
we propose a minimization-free-iteration algorithm.
For this aim, we define the function $\overline{\cal G}_\kappa$ as
\begin{align}
\overline{\cal G}_\kappa(\eta^{(d_0)}):=&
\sum_{j=1}^{d_0} \eta_j \overline{\Omega}^j
(\eta^{(d_0)})
+\overline{\Omega}^{d_0+1}(\eta^{(d_0)})
\Label{form9},\\
\overline{\Omega}^j(\eta^{(d_0)}):=&
\tilde{\Omega}^j(\eta^{(d_0)},c),\Label{ZXC1}\\
\overline{\Omega}^{d_0+1}(\eta^{(d_0)}):=&
\sum_{j=d_0+1}^{d'-1} c_j \tilde{\Omega}^j(\eta^{(d_0)},c)
+\tilde{\Omega}^{d'}(\eta^{(d_0)},c).\Label{ZXC2}
\end{align}
In the following, we use the natural parameter
$\theta_{(d_0)} \in \mathbb{R}^{d_0}$.
Then, we recall 
the convex function
$\phi_\kappa(\theta_{(d_0)})$, i.e.,
the special case studied in Example \ref{Ex1}.
Our mixture family is given as
\begin{align}
{\cal M}_\kappa
:= \Big\{\theta \in \mathbb{R}^{d_0+1}\Big| 
\frac{\partial \phi_\kappa}{\partial \theta^{d_0+1}}
=1 \Big\}.\Label{BNB}
\end{align}

Since
\begin{align}
&\tilde{\cal G}(\eta^{(d_0)},
c)=
\overline{\cal G}_\kappa(\eta^{(d_0)}), 
\Label{OY1}
\end{align}
the minimization \eqref{BNY} can be written 
as $\min_{\eta^{(d_0)}\in \mathbb{R}^{d_0}}\overline{\cal G}_\kappa(\eta^{(d_0)})$.
Hence, we discuss the minimization 
$\min_{\eta^{(d_0)} \in \mathbb{R}^{d_0}} \overline{\cal G}_\kappa(\eta^{(d_0)})$
by using the algorithm with
the Bregman divergence system 
based on the convex function $\phi_\kappa$ defined by \eqref{BNX}.
In addition, we have the following lemma.
\begin{lemma}
When any two elements $P_X ,Q_X\in {\cal M}_p$ satisfy
\begin{align}
\sum_{x \in \cX}P_X(x) (\Omega_p[P_X](x)-\Omega_p[Q_X](x))
\ge 0,
\end{align}
any elements 
$\eta^{(d_0)}, {\eta'}^{(d_0)}\in \mathbb{R}^{d_0}$ satisfy
\begin{align}
D_{\overline{\Omega}}((\eta^{(d_0)},1)\| 
({\eta'}^{(d_0)},1))
\ge 0.\Label{MCY}
\end{align}
\end{lemma}

\begin{proof}
Any element of ${\cal M}_p$ is written as $\tilde{P}_{
(\eta^{(d_0)},c)}$.
Since \eqref{ZXC1} and \eqref{ZXC2}
imply
$\sum_{x \in \cX}\tilde{P}_{
(\eta^{(d_0)},c)}(x) (\Omega_p[\tilde{P}_{
(\eta^{(d_0)},c)}](x)-\Omega_p[\tilde{P}_{
({\eta'}^{(d_0)},c)}](x))
=D_{\overline{\Omega}}((\eta^{(d_0)},1)\|( {\eta'}^{(d_0)},1))
$, 
we have \eqref{MCY}.
\end{proof}

Then, for $\theta_{(d_0+1)} \in \mathbb{R}^{d_0+1}$,
we have
\begin{align}
\frac{\partial \phi_\kappa}{\partial \theta ^{j}}(\theta_{(d_0+1)})
&=\sum_{x \in \cX} f_j(x) e^{\sum_{j=1}^{d_0}f_j(x) \theta^j + \theta^{d_0+1}} ,\\
\frac{\partial \phi_\kappa}{\partial \theta ^{d_0+1}}(\theta_{(d_0+1)})
&=\sum_{x \in \cX} e^{\sum_{j=1}^{d_0}f_j(x) \theta^j + \theta^{d_0+1}} \notag\\
&=\bigg(\sum_{x \in \cX} e^{\sum_{j=1}^{d_0}f_j(x) \theta^j}\bigg)
e^{\theta^{d_0+1}} 
\end{align}
for $j=1, \ldots, d_0$.
Hence, for $(\theta_{(d_0)},\theta^{d_0+1}) \in \mathbb{R}^{d_0+1}$,
the $e$-projection 
$\Pro^{(e),\phi_\kappa}_{\mathcal{M}} 
(\theta_{(d_0)},\theta^{d_0+1})$ 
is given as
\begin{align}
(\Pro^{(e),\phi_\kappa}_{\mathcal{M}} (
(\theta_{(d_0)},\theta^{d_0+1}))^j =&\theta^j, \\
(\Pro^{(e),\phi_\kappa}_{\mathcal{M}} 
((\theta_{(d_0)},\theta^{d_0+1}))^{d_0+1} =&
u(\theta_{(d_0)})\notag\\
:=&
-\log 
\bigg(\sum_{x \in \cX} e^{\sum_{j=1}^{d_0}f_j(x) \theta^j }
\bigg)
\end{align}
for $j=1, \ldots, d_0$.
Hence, as discussed in Example \ref{Ex1},
the $e$-projection 
$\Pro^{(e),\phi_\kappa}_{\mathcal{M}} 
(\theta_{(d_0)},\theta^{d_0+1}) $ can be easily calculated.
Since Algorithm \ref{AL1} can be done by the pair of 
\begin{align}
\eta^{(d_0)}_{[t]}&:=
(\frac{\partial \phi_\kappa}{\partial \theta ^{j}}
(\theta_{(d_0)}^{[t]},u(\theta_{(d_0)}^{[t]})))_{j=1}^{d_0}, \\
\theta_{(d_0)}^{[t+1]}&:=
\theta_{(d_0)}^{[t]}  - \frac{1}{\gamma} 
\overline{\Omega}(\eta^{(d_0)}_{[t]}),
\end{align}
Algorithm \ref{AL1} is simplified as 
Algorithm \ref{AL3}.

\begin{algorithm}
\caption{BD-based AB algorithm for $\min_{P \in {\cal M}_p} \tilde{\cal G}_p(P)$
with $\phi_\kappa$}
\Label{AL3}
\begin{algorithmic}
\STATE {Choose the initial value 
$\theta_{(d_0)}^{[1]} \in \mathbb{R}^{d_0}$;} 
\REPEAT 
\STATE Calculate 
\begin{align}
\theta_{(d_0)}^{[t+1]}:=
\theta_{(d_0)}^{[t]}  - \frac{1}{\gamma} 
\overline{\Omega}
\Big(\frac{\partial \phi_\kappa}{\partial \theta ^{j}}
\big(\theta_{(d_0)}^{[t]},
u(\theta_{(d_0)}^{[t]})\big)_{j=1}^{d_0}\Big);
\Label{ITE1}\end{align}
\UNTIL{convergence;} 
\STATE{We denote the convergent of 
$\eta^{(d_0)}_{[t]}$ by $\eta^{(d_0)}_{[\infty]}$.}
\STATE{\bf output} $\overline{\cal G}_\kappa(\eta^{(d_0)}_{[\infty]})$.
\end{algorithmic}
\end{algorithm}

However,
$D^{\phi_\kappa^*}(({\eta'}^{(d_0)},1)\|({\eta}^{(d_0)},1))$
does not coincide with 
$D^{\phi^*}(({\eta'}^{(d_0)}, c,1)\|
({\eta}^{(d_0)}, c,1))$, in general.
Therefore, the algorithm in the previous paper \cite{Hmixture}
does not coincide with our algorithm with 
the objective function $ \tilde{\cal G}$ defined in \eqref{OY1}
and the Bregman Divergence system 
based on the convex function $\phi_\kappa$ defined by \eqref{BNX}, in general.

Due to the above difference,
there is a risk that 
$\tilde{P}_{(\eta_{[t]}, c)}$
has negative components.
To cover this case, we need to 
define 
$\overline{\Omega}^j(\eta^{(d_0)})$
for $j=1, \ldots,d_0+1$ even in this case.
In the above discussion, we assume that 
the terms $\overline{\Omega}^j(\eta^{(d_0)})$
are defined even when $\tilde{P}_{(\eta_{[t]}, c)}$
has negative components.
The method for the extension of the definition depends on the problem setting.
We discuss this problem only in several special cases
in the latter sections.

\section{Application to em-algorithm}\Label{S6}
The em-algorithm is an algorithm to calculate the minimum divergence between an exponential family ${\cal E}_p$ and a mixture family ${\cal M}_p$ over a sample space ${\cal X}$.
The aim is to calculate the following minimum value:
\begin{align}
\min_{P \in {\cal M}_p,Q \in {\cal E}_p} 
D(P\|Q).
\end{align}
This problem appears in Boltzmann machine \cite{AKN}, rate-distribution \cite{H23}.
In the em-algorithm, we alternatively apply 
the $e$-projection $\Pro^{(e)}_{\mathcal{M}_p}$
and
the $m$-projection $\Pro^{(m)}_{\mathcal{E}_p}$.
In the case of probability distributions,
the $m$-projection $\Pro^{(m)}_{\mathcal{E}_p}$
is written as an affine map,
but 
the $e$-projection $\Pro^{(e)}_{\mathcal{M}_p}$
has a more complicated form.
When $k$ constraints define the mixture family $\mathcal{M}_p$,
the $e$-projection $\Pro^{(e)}_{\mathcal{M}_p}$
is given as a convex minimization with $k$ variables.
Therefore, 
the $e$-projection $\Pro^{(e)}_{\mathcal{M}_p}$
is the bottleneck in the em-algorithm.

To avoid the above convex minimization,
we employ 
the Bregman-divergence-based Arimoto-Blahut algorithm based on the convex function $\phi_\kappa$ defined in \eqref{BNX}.
For this aim, we assume that 
the mixture family $\mathcal{M}_p$ is defined in
the same way as Section \ref{S5}.
Then, we have
\begin{align}
&\min_{P \in {\cal M}_p,Q \in {\cal E}_p} 
D(P\|Q)=
\min_{P \in {\cal M}_p}
D\big(P \big\|\Pro^{(m)}_{\mathcal{E}_p}(P)\big)\notag \\
=&
\min_{P \in {\cal M}_p}
\sum_{x \in {\cal X}}
P(x) \big(\log P(x)
-\log \Pro^{(m)}_{\mathcal{E}_p}(P) (x)\big).\Label{BNQ}
\end{align}

Now, we discuss how to define
$\overline{\Omega}^j(\eta^{(d_0)})$
for $j=1, \ldots,d_0+1$ even 
when $\tilde{P}_{(\eta^{[t]}, c)}$
has negative components.
Since the $m$-projection 
$\Pro^{(m)}_{\mathcal{E}_p}$
is written as an affine map,
it can be naturally extended to the above negative case.
Hence,
for $\eta \in \mathbb{R}^{d_0}$, we can define
$\Pro^{(m)}_{\mathcal{E}_p}(\tilde{P}_{(\eta^{[t]}, c)})$
while it potentially has negative components.
Then, using the notation $(x)_+:=\max(x,\epsilon)$
with very small $\epsilon>0$,
for $\eta \in \mathbb{R}^{d_0}$,
we define 
\begin{align}
&\overline{\Omega}^j(\eta) \notag\\
:=&
\sum_{x\in \cX}g^j(x) 
\Big(\log (\tilde{P}_{(\eta,c)}(x))_+
-\log (\Pro^{(m)}_{\mathcal{E}_p}(\tilde{P}_{(\eta,c)}) (x))_+\Big),\\
&\overline{\Omega}^{d_0+1}(\eta) \notag\\
:=&
\sum_{j=d_0+1}^{d-1} c_j
\sum_{x\in \cX}g^j(x) 
\Big(\log (\tilde{P}_{(\eta,c)}(x))_+\notag\\
&\hspace{20ex} -\log (\Pro^{(m)}_{\mathcal{E}_p}(\tilde{P}_{(\eta,c)}) (x))_+\Big)\notag\\
&+
\sum_{x\in \cX}g^d(x) 
(\log (\tilde{P}_{(\eta,c)}(x))_+
-\log (\Pro^{(m)}_{\mathcal{E}_p}(\tilde{P}_{(\eta,c)}) (x))_+).
\end{align}
Then, we define $\overline{\cal G}_\kappa(\eta)$
as \eqref{form9}.
When the components of 
$\tilde{P}_{(\eta,c)}$
are greater than $\epsilon$,
$\overline{\cal G}_\kappa(\eta)$ equals the objective function in \eqref{BNQ} with $P=\tilde{P}_{(\eta,c)}$.
Hence, when $\epsilon$ is sufficiently small,
the minimum of 
$\overline{\cal G}_\kappa(\eta)$ equals
the minimum \eqref{BNQ}.
Then, we can apply Algorithm \ref{AL3}.
Since Algorithm \ref{AL3} has
no convex minimization,
the presented method does not need to solve
any convex minimization.

\section{Application to rate-distortion theory}\Label{S7}
\subsection{Theoretical analysis}\Label{S7-1}
We consider two systems ${\cal X}=
\{1, \ldots, d_1\}$ and ${\cal Y}=
\{1, \ldots, d_2\}$.
Given a distortion function $R(x,y)$ and a distribution
$P_X$ on ${\cal X}$, 
the optimal compression rate in rate-distortion theory 
is formulated as the following minimization problem.
\begin{align}
\min_{W_{Y|X}}\bigg\{I(X;Y)_{W_{Y|X}\times P_X}\bigg|
\sum_{x,y}(W_{Y|X}\times P_X)(y,x) R(x,y)=c
\bigg\},\label{BN2}
\end{align}
where
$W_{Y|X}$ is a conditional distribution on $\mathcal{Y}$ for all $X \in \mathcal{X}$,
and $W_{Y|X}\times P_X$ is defined as the joint distribution as
\begin{align*}
(W_{Y|X}\times P_X)(y,x):=
W_{Y|X}(y|x)P_X(x).
\end{align*}
In this setting,
the distribution $P_X$ is fixed and the conditional distribution
$W_{Y|X}$ needs to be optimized to minimizing the mutual information $I(X;Y)_{W_{Y|X}\times P_X}$.
We consider the exponential family
${\cal E}_p:=
\{P_X \times Q_Y| Q_Y \in {\cal P}({\cal Y})\}$.
Then, we have
\begin{align}
I(X;Y)_{W_{Y|X}\times P_X}=
D\big(W_{Y|X}\times P_X\big\| 
\Pro^{(m)}_{\mathcal{E}_p} (W_{Y|X}\times P_X)\big).
\end{align}
Hence, the minimization \eqref{BN2} is a special case of 
the problem in the previous section \cite{H23}.
Here, we discuss how to apply our algorithm given in Section \ref{S5} to this example in a more precise way.

We assume that $R(d_1,d_2)\neq R(d_1,d_2-1)$.
The set 
$\Big\{W_{Y|X}\times P_X\Big|
\sum_{x,y}W_{Y|X}\times P_X(y,x) R(x,y)=c
\Big\}$ is a mixture family.
Also, $I(X;Y)_{W_{Y|X}\times P_X}$ is written as
$
\sum_{x}P_X(x)\sum_{y}W_{Y|X}(y|x)
(\log W_{Y|X}(y|x)-
\log (\sum_{x'}P_X(x') W_{Y|X}(y|x'))
)$.
To calculate \eqref{BN2},
we define functions 
$f_{(i-1)(d_2-1)+j}(x,y):=\delta_{i}(x)\delta_j(y)$ 
for $i=1,\ldots,d_1-1$
and $j=1,\ldots,d_2-1$.
$f_{(d_1-1)(d_2-1)+j}(x,y):=\delta_{d_1}(x)\delta_j(y)$
for $j=1,\ldots,d_2-2$.
We choose its dual functions as 
\begin{align}
&g^{(i-1)(d_2-1)+j}(x,y)\notag\\
:=&
\delta_{i}(x)(\delta_j(y)-\delta_{d_2}(y))\notag\\
&-\frac{R(i,j)- R(i,d_2)}{R(d_1,d_2-1)- R(d_1,d_2)}
\delta_{d_1}(x)(\delta_{d_2-1}(y)-\delta_{d_2}(y))
\end{align}
for $i=1,\ldots,d_1-1, 
j=1,\ldots,d_2-1$, and
\begin{align}
&g^{(d_1-1)(d_2-1)+j}(x,y)\notag\\
:=&
\delta_{d_1}(x)(\delta_j(y)-\delta_{d_2}(y))\notag\\
&-
\frac{R(d_1,j)- R(d_1,d_2)}{R(d_1,d_2-1)- R(d_1,d_2)}
\delta_{d_1}(x)(\delta_{d_2-1}(y)-\delta_{d_2}(y))
\end{align}
 for $j=1,\ldots,d_2-2$.
Setting $d_0= d_1(d_2-1)-1$,
we have
\begin{align}
\sum_{x,y} f_{i}(x,y)g^j(x,y)&=\delta_{i,j},\\
\sum_{x,y} R(x,y)g^j(x,y)&=0,\\
\sum_{x,y} \delta_i(x)g^j(x,y)&=0, \\
\sum_{x,y} f_{i'}(x,y)\delta_{i'}(x) \delta_{d_2}(y)&=0,\\
\sum_{x,y} f_{i'}(x,y)\delta_{d_1}(x) \delta_{d_2-1}(y)&=0
\end{align}
for $i,j=1\dots,d_0$ and $i'=1, \ldots, d_1$.

Then, we choose a mixture parameter
$\eta_j:= \sum_{x,y}P(x,y) f_j(x,y)$ for $j=1, \ldots, d_0$.
We write the distribution corresponding to the mixture parameter $\eta$ by 
$W_{Y|X|\eta} \times P_X $.
Then, we have
\begin{align}
&\sum_{x}P_X(x)\sum_{y}W_{Y|X|\eta}(y|x)
\Big(\log W_{Y|X|\eta}(y|x)\notag\\
&-
\log \big(\sum_{x'}P_X(x') W_{Y|X|\eta}(y|x')\big)\Big)\notag\\
=&
\sum_{x,y}
\Big(\sum_{j}\eta_j g^j(x,y)
+P_X(x)\delta_{d_2}(y)\notag\\
&
+\frac{c-\sum_{x}P_X(x)R(x,d_2)}{R(d_1,d_2-1)- R(d_1,d_2)}
\delta_{d_1}(x) 
(\delta_{d_2-1}(y)-\delta_{d_2}(y))
\Big) \notag\\
&\cdot\Big(\log W_{Y|X|\eta}(y|x)-
\log \big(\sum_{x'}P_X(x') W_{Y|X|\eta}(y|x')\big)\Big)\notag\\
=&
\sum_{j}\eta_j 
\sum_{x,y} g^j(x,y)
\cdot\Big(\log W_{Y|X|\eta}(y|x)\notag\\
&-
\log \big(\sum_{x'}P_X(x') W_{Y|X|\eta}(y|x')\big)\Big)\notag\\
&+\sum_{x,y}
\Big(P_X(x)\delta_{d_2}(y)
+
\frac{c-\sum_{x}P_X(x)R(x,d_2))}{R(d_1,d_2-1)- R(d_1,d_2)}
\delta_{d_1}(x) \Big)
\notag\\
&\cdot\Big(\log W_{Y|X|\eta}(y|x)-
\log \big(\sum_{x'}P_X(x') W_{Y|X|\eta}(y|x')\big)\Big).
\end{align}

In this case, the joint distribution 
$W_{Y|X|\eta} \times P_X (x,y)$ is written as the following distribution:
\begin{align}
&P_\eta(x,y)\notag\\
:=&\Big(\sum_{j}\eta_j g^j(x,y)
+P_X(x)\delta_{d_2}(y)\notag\\
&
+\frac{c-\sum_{x'}P_X(x')R(x',d_2)
}{R(d_1,d_2-1)- R(d_1,d_2)}
\delta_{d_1}(x) 
(\delta_{d_2-1}(y)-\delta_{d_2}(y))
\Big).\Label{BND}
\end{align}
Then, using the notation $(x)_+:=\max(x,\epsilon)$
with very small $\epsilon>0$,
we set
\begin{align}
&\tilde{\Omega}^j(\eta)\notag\\
:=&\sum_{x,y}
g^j(x,y)
\cdot\Big(\log \big(P_\eta(x,y)\big)_+ -\log P_X(x)\notag\\
&\qquad -\log 
\Big(\sum_{x'\in {\cal X}}P_\eta(x',y)\Big)_+\Big),\\
&\tilde{\Omega}^{d_0+1}(\eta)\notag\\
:=&\sum_{x,y}
\Big(P_X(x)\delta_{d_2}(y)
+
\frac{c-\sum_{x}P_X(x)R(x,d_2))}{R(d_1,d_2-1)- R(d_1,d_2)}
\delta_{d_1}(x) \Big)\notag\\
&\cdot\Big(\log \big(P_\eta(x,y) \big)_+ -\log P_X(x)-\log 
\Big(\sum_{x'\in {\cal X}}P_\eta(x',y)\Big)_+\Big).
\end{align}
In the definition \eqref{BND}, there is a possibility that 
one of the values 
$\{P_\eta(x,y)\}_{x,y}$ is a negative value.
Hence, we define 
$\tilde{\Omega}^j(\eta)$ and 
$\tilde{\Omega}^{d_0+1}(\eta)$ in the above way.
We have
\begin{align}
&\overline{\cal G}_\kappa(\eta)\notag\\
=&
\sum_{x,y}
P_\eta(x,y) \log (P_\eta(x,y))_+ 
-\sum_{x}P_X(x)\log P_X(x)
\notag\\
&-\sum_{y} \big(\sum_{x'} P_\eta(x',y)\big)\log 
\Big(\sum_{x'} P_\eta(x',y)\Big)_+,
\end{align}
and
\begin{align}
\sum_{x,y:P_\eta(x,y)<0}|P_\eta(x,y)|
\ge 
\sum_{y:(\sum_{x'} P_\eta(x',y))<0}
\Big|\Big(\sum_{x'} P_\eta(x',y)\Big)\Big|.\Label{BNI}
\end{align}
When the equality does not hold in the inequality \eqref{BNI},
$\overline{\cal G}_\kappa(\eta)\to \infty$ as $\epsilon \to 0$.
When one of the values 
$\{P_\eta(x,y)\}_{x,y}$ is a negative value,
the function $\overline{\cal G}_\kappa(\eta)$ takes a very large number
with a sufficient small $\epsilon>0$
in most cases.
Therefore, we can expect that 
the minimization with this function avoids the case when 
one of the values 
$\{P_\eta(x,y)\}_{x,y}$ is a negative value.

\subsection{Conventional method}\Label{S5S1}
Next, we discuss the conventional case for rate-distortion theory.
That is,
the Bregman divergence is given as the KL divergence of probability distributions.
In this case, using 
the one-variable smooth convex function
$\hat{F}[P_{Y}](\tau) :=$ $\sum_{x}P_{X}(x)$ $
\log \Big(\sum_{y} P_{Y}(y) e^{{\tau}(D- d(x,y)))}\Big)
$,
the reference \cite[Algorithm 9]{H23} proposed Algorithm \ref{protocol1-2}, which
corresponds to Algorithm \ref{AL1} with $\gamma=1$
when the Bregman divergence is given as KL divergence.

\begin{algorithm}
\caption{em-algorithm for rate distortion}
\Label{protocol1-2}
\begin{algorithmic}
\STATE {Choose the initial distribution $P_Y^{(1)}$ on $\Y$.
Then, we define the initial joint distribution $P_{XY,(1)}$ as 
$P_Y^{(1)}\times P_X $;} 
\REPEAT 
\STATE {\bf m-step:}\quad 
Calculate $P_{Y|X}^{(t+1)}$ as 
$P_{Y|X}^{(t+1)}(y|x):=
P_{Y}^{(t)}(y)e^{\bar{\tau} d(x,y)}
\Big(\sum_{y'}P_{Y}^{(t)}(y') e^{\bar{\tau} d(x,y')} \Big)^{-1}$,
where $\bar{\tau}$ is given as
$\argmin_{\tau} 
\hat{F}[P_{Y}^{(t)}](\tau)$.
\STATE {\bf e-step:}\quad 
Calculate $P_{Y}^{(t+1)}(y)$ as 
$\sum_{x \in \X}P_{Y|X}^{(t+1)}(y|x)P_X(x)$.
\UNTIL{convergence.} 
\end{algorithmic}
\end{algorithm}
However,
in the realistic case, we need to care about the error of the minimization
in the $m$-step.
To address this problem,
we need to clarify what algorithm to be used for the convex minimization.
As a typical one, employing the Newton method, we revise Algorithm \ref{protocol1-2} as Algorithm \ref{protocol1-3}.

\begin{algorithm}
\caption{em-algorithm for rate distortion with the Newton method}
\Label{protocol1-3}
\begin{algorithmic}
\STATE {Choose the initial distribution $P_Y^{(1)}$ on $\Y$.
Then, we define the initial joint distribution $P_{XY,(1)}$ as 
$P_Y^{(1)}\times P_X $;} 
\REPEAT 
\STATE {\bf m-step:}\quad 
Set $\tau_0=0$.
\REPEAT 
\STATE
Set 
\begin{align}
\tau_k=\tau_{k-1}-\frac{\frac{d}{d\tau}\hat{F}[P_{Y}^{(t)}](\tau_{k-1}) }{
\frac{d^2}{d\tau^2}\hat{F}[P_{Y}^{(t)}](\tau_{k-1})}.
\Label{ITE2}
\end{align}
\UNTIL{$k=f(t)$.} 
\par
Set $\bar{\tau}$ as the above outcome.
Calculate $P_{Y|X}^{(t+1)}$ as $P_{Y|X}^{(t+1)}(y|x):=
P_{Y}^{(t)}(y)e^{\bar{\tau} d(x,y)}
\Big(\sum_{y'}P_{Y}^{(t)}(y') e^{\bar{\tau} d(x,y')} \Big)^{-1}$.
\STATE {\bf e-step:}\quad 
Calculate $P_{Y}^{(t+1)}(y)$ as 
$\sum_{x \in \X}P_{Y|X}^{(t+1)}(y|x)P_X(x)$.
\UNTIL{convergence.} 
\end{algorithmic}
\end{algorithm}

\subsection{Numerical analysis for classical rate distortion without side information}\Label{S5S1-R}
To see how our algorithm works, we numerically compare
our algorithm with the algorithm by \cite[Section V-C]{H23}.
For this aim, we choose the same example as \cite[Section V-C]{H23}, i.e., we focus on
the case when $d_1=d_2=3$, $c=1.5$, and 
the cost function $R$ is chosen as
\begin{align}
\left(
\begin{array}{ccc}
d(1,1) & d(1,2) & d(1,3)
\\
d(2,1) & d(2,2) & d(2,3)
\\
d(3,1) & d(3,2) & d(3,3)
\end{array}
\right)
=\left(
\begin{array}{ccc}
0 & 1 & 2\\
1 & 2 & 0\\
3 & 0 & 1
\end{array}
\right),
\end{align}
and the distribution $P_X$ is chosen as
\begin{align}
P_X(1)=0.5,~
P_X(2)=0.3,~
P_X(3)=0.2.
\end{align}

In this case, the application of the algorithm by \cite{H23}
guarantees that 
the minimum mutual information $I(X;Y)$ is
\begin{align}
I(X;Y)_{P_{XY}^{*}}:= 0.100039,
\end{align}
and it is attained by the conditional distribution given as
\begin{align}
P_{Y|X}^{*}=
\left(
\begin{array}{ccc}
0.0855598 &0.188594& 0.430983 \\
0.22431&0.494433 &0.139579 \\
0.69013& 0.316974&0.429438
\end{array}
\right).
\end{align}

In the above example,
we compare Algorithm \ref{protocol1-3} 
 and Algorithm \ref{AL3} with the choices given in Subsection \ref{S7-1}.
In Algorithm \ref{protocol1-3}, we choose 
$f_1(t)=5 +t$ and $f_2(t)=\lceil 5+3\log t \rceil$.
The mutual information in the numerical calculation in \cite{H23} achieves 
$0.100039$.
In Algorithm \ref{AL3} with the choices given in Subsection \ref{S7-1},
we set the initial parameter $\theta^{[1]}$,
$\epsilon$, and $\gamma$
 to be $(0,0,0,0,0)$, $0.0001$, and $50$, respectively.
For the comparison, 
counting the iterations in the Newton method, 
we consider that $t$-th step in Algorithm \ref{protocol1-3} has 
the cumulative number of iterations 
$\sum_{k=1}^t f_i(k)$ with $i=1,2$.
Based on the above idea, we made a numerical calculation as Fig. \ref{Error}.
This comparison shows the advantage of the method presented in Subsection \ref{S7-1} in the initial phase.

\begin{figure}[htbp]
\begin{center}
 \includegraphics[width=0.95\linewidth]{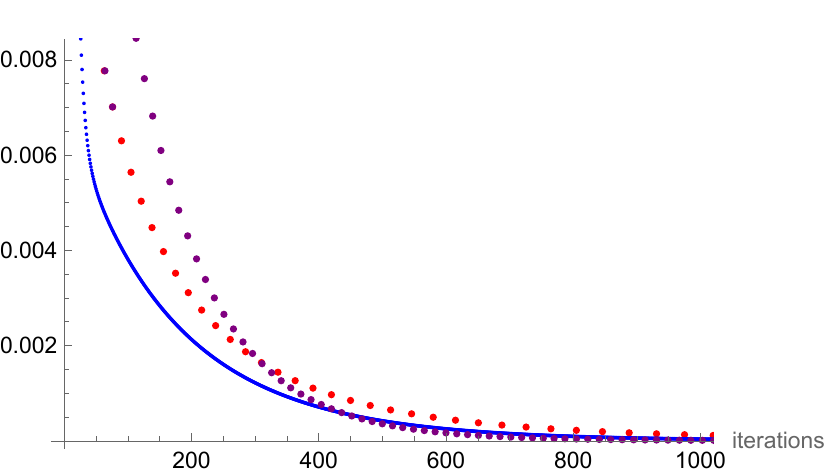}
  \end{center}
\caption{Behavior of  
$\tilde{\cal G}(\theta^{[t]})-\tilde{\cal G}(\theta^{[\infty]})$
of the minimum mutual information.
Vertical axis shows the value of $\tilde{\cal G}(\theta^{[t]})-\tilde{\cal G}(\theta^{[\infty]})$. 
The horizontal axis shows the number of iterations.
The red points show 
the number of iterations of the calculation \eqref{ITE2}
in Algorithm \ref{protocol1-3}
with $f_1$.
The purple points show the same number in Algorithm \ref{protocol1-3}
with $f_2$}.
The blue points shows 
the number of iterations of the calculation \eqref{ITE1}
in Algorithm \ref{AL3} with the choices given in Subsection \ref{S7-1}.
\Label{Error}
\end{figure}   

\section{Application to quantum states with linear constraints}\Label{S8}
We apply our method
to the case when ${\cal M}$ is given as a mixture family of quantum states, i.e., 
a set of quantum states with linear constraints.
The contents of this section is a quantum version of the contents of Section \ref{S5}.

We consider a finite-dimensional, i.e., $\overline{d}$-dimensional system
${\cal H}$.
We introduce $\overline{d}^2-1$ Hermitian matrices 
$A_1,\ldots, A_{\overline{d}^2-1}$ over $\cH$
such that 
$A_{1}, \ldots, A_{\overline{d}^2-1}$ and $I$ 
are linearly independent.
We also choose 
$\overline{d}^2$ Hermitian matrices 
$B^1, \ldots, B^{\overline{d}^2}$ on $\cX$ such that
\begin{align}
\Tr A_{j}B^i=\delta^i_{j},
\end{align}
where $A_{\overline{d}^2}:=I$.
We define a mixture family ${\cal M}_q$
as the set of states to satisfy the linear constraints
\begin{align}
\Tr \rho A_j =c_j \hbox{ for }j=d_0+1, \ldots, \overline{d}^2-1.\Label{BNI2}
\end{align}
Then, we have two kinds of parametrization of the state  on $\cH$.
Using the natural parameter $\theta \in \mathbb{R}^{\overline{d}^2-1}$,
we parameterize the distribution as
$\rho_\theta:= 
e^{\sum_{j=1}^{\overline{d}-1} A_j \theta^j
-\theta^{\overline{d}^2}(\theta)}$,
where $\theta^{\overline{d}^2}(\theta):=-\log 
\Tr e^{\sum_{j=1}^{\overline{d}-1} A_j \theta^j}$.
Using the mixture parameter $\eta \in \mathbb{R}^{\overline{d}^2-1}$,
we parameterize the state as
$\tilde{\rho}_\eta :=\sum_{j=1}^{\overline{d}^2-1}
\eta_j B^j+ B^{\overline{d}^2}$.
The state $\tilde{\rho}_\eta$ belongs to ${\cal M}_q$
if and only if
$\eta_{j}=c_j$ for $j=d_0+1, \ldots, \overline{d}^2-1$.

Then, we introduce a Hermitian matrix $\Omega_q[\rho]$ 
on $\cH$, which depends on 
the state $\rho \in {\cal M}_q$.
We define an objective function
${\cal G}_q(\rho):=\Tr \rho \Omega_q[\rho]$, 
and
consider the minimization problem
\begin{align}
\min_{\rho \in {\cal M}_q} {\cal G}_q(\rho).\Label{BNY2}
\end{align}
Under the condition
\begin{align}
\Tr \rho (\Omega_q[\rho]-\Omega_q[\sigma])
\le \gamma D(\rho\|\sigma),
\end{align}
the paper \cite{HL} introduced an algorithm, where
$D(\rho\|\sigma)$ is the quantum relative entropy
$\Tr \rho (\log \rho -\log \sigma)$.
The previous algorithm \cite{HL} coincides with our algorithm 
under the following choices.

We define the function $\phi(\theta,\theta^{\overline{d}^2})$ 
on $\mathbb{R}^{\overline{d}^2}$ as 
\begin{align}
\phi(\theta,\theta^{\overline{d}^2})
=\Tr e^{\sum_{j=1}^{\overline{d}^2-1} A_j \theta^j
+\theta^{\overline{d}^2}} .
\Label{BCR2}
\end{align}
We choose the mixture family ${\cal M}$ as
\begin{align}
{\cal M} 
:= \Bigg\{(\theta,\theta^{d^2}) 
\in \mathbb{R}^{\overline{d}^2}
\Bigg| 
\begin{array}{l}
\frac{\partial \phi}{\partial \theta^j}(\theta,\theta^{\overline{d}^2})
=c_j,\\ 
\frac{\partial \phi}{\partial 
\theta^{\overline{d}^2}}(\theta,\theta^{\overline{d}^2})
=1 
\end{array}
\Bigg\},
\end{align}
where the index $j$ in the above condition
runs from $d_0+1$ to $\overline{d}^2-1$.

The relation $\rho \in {\cal M}_q$ holds 
if and only if $(\theta,\theta^{\overline{d}^2}(\theta)) \in {\cal M}$.
The relation $\tilde{\rho}_\eta \in {\cal M}_q$ holds 
if and only if $
(\eta,1) \in \Xi_{\cal M}$.

For 
$\eta,\eta'$ such that 
$(\eta,1),(\eta',1)\in \Xi_{\cal M} $,
we choose
\begin{align}
\tilde{\Omega}^j(\eta):=
\Tr B^j \Omega_q[\tilde{\rho}_\eta]
\hbox{ for }j=1, \ldots, \overline{d}^2
\end{align}
and have
\begin{align}
&{\cal G}_q(\tilde{\rho}_\eta)
= \tilde{\cal G}(\eta):=
\sum_{j=1}^{\overline{d}^2-1} \eta_j\tilde{\Omega}^j(\eta) 
+\tilde{\Omega}^{\overline{d}^2}(\eta) 
=\Tr\tilde{\rho}_\eta \Omega_q[\tilde{\rho}_{\eta}],
\Label{O1Q}\\
&D(\tilde{\rho}_\eta\|\tilde{\rho}_{\eta'})=
D^{\phi^*}((\eta',1)\|(\eta,1)).
\Label{O3Q}
\end{align}
Therefore, the algorithm in the previous paper \cite{HL}
coincides with our algorithm with 
the objective function $ \tilde{\cal G}$ defined in \eqref{O1Q}
and the Bregman divergence system 
based on the convex function $\phi$ defined by \eqref{BCR2}.
However, the previous algorithm and our algorithm of the above choice 
have the process for the $e$-projection 
$\Pro^{(e),\phi}_{\mathcal{M}} (\theta) $.
Its step needs to solve a convex minimization with 
$\overline{d}^2-d_0-1$
parameters. 
This problem can be resolved by the same method as Section \ref{S5}.
The papers \cite{RISB,HY,HL}
explain concrete choices of $\Omega$ and linear constraints
including classical-quantum channel coding,
information bottleneck.
In particular, 
when 
an exponential family
${\cal E}:=
\{\rho_\theta\}_{\theta}$ with 
$\rho_\theta:=\exp(\sum_{j=1}^k \theta^j Y_j)
/\Tr \exp(\sum_{j=1}^k \theta^j Y_j)$ is given,
where $Y_j$ is an Hermitian matrix,
we often consider the minimum divergence
between the exponential family ${\cal E}$
and the mixture family ${\cal M}$;
\begin{align}
&\min_{\rho \in {\cal M}}
\min_{\sigma \in {\cal E}}D(\rho\|\sigma)
=
\min_{\rho \in {\cal M}}
D(\rho\| \Gamma_{\cal E}^{(m)}(\rho))
\notag \\
=&\min_{\rho \in {\cal M}}
\Tr (\log \rho -\log \Gamma_{\cal E}^{(m)}(\rho)),
\end{align}
where
\begin{align}
\Gamma_{\cal E}^{(m)}(\rho):=
\argmin_{\sigma \in {\cal E}}D(\rho\|\sigma).
\end{align}
This problem including quantum rate-distortion theory has been studied 
by using quantum em algorithm in \cite{H23}.
In contrast, the paper \cite{HSF} applies the mirror descent method to quantum rate-distortion theory. 
Our method can be applied to this problem as well.

\section{Analysis for our general algorithm and Proof of Theorem \ref{NLT}}\Label{S9}
Indeed, Algorithm \ref{AL1} is characterized as the iterative minimization of 
the following two-variable function, i.e., the extended objective function;
\begin{align}
J_\gamma(\theta,\theta'):=\gamma D^\phi(\theta\|\theta')+
\sum_{j=1}^d \eta_j(\theta) \Omega^j(\theta').
\end{align}
To see this fact, 
as a generalization of a part of \cite[Lemma 3.2]{RISB}, $
\min_{\theta \in \Theta}  J_\gamma(\theta,\theta')
$ is calculated as follows.
\begin{lemma}\Label{L1}
We have $
\argmin_{\theta \in \Theta}  J _\gamma(\theta,\theta')=
\Pro^{(e),\phi}_{\mathcal{M}} \circ{\cal F}_\gamma(\theta')$, i.e., 
\begin{align}
&\min_{\theta \in \Theta}  J_\gamma(\theta,\theta')=
J _\gamma \Big(\Pro^{(e),\phi}_{\mathcal{M}} \circ{\cal F}_\gamma(\theta'),\theta'\Big) \nonumber \\
=&
\gamma \Big(
D^\phi(\Pro^{(e),\phi}_{\mathcal{M}} \circ{\cal F}_\gamma(\theta')\| {\cal F}_\gamma(\theta')) 
+ \phi(\theta') - \phi({\cal F}_\gamma(\theta'))
\Big)
,\Label{XMY} \\
& J _\gamma(\theta,\theta')
 =\min_{\tilde{\theta} \in \Theta}  J _\gamma(\tilde{\theta},\theta')
+\gamma D^\phi(\theta\| \Pro^{(e),\phi}_{\mathcal{M}} \circ{\cal F}_\gamma(\theta')) \Label{XMY2UU}  \\
=&J _\gamma\big(\Pro^{(e),\phi}_{\mathcal{M}} \circ{\cal F}_\gamma(\theta'),\theta' \big)
+\gamma D^\phi(\theta\| \Pro^{(e),\phi}_{\mathcal{M}} \circ{\cal F}_\gamma(\theta')).
\Label{XMY2} 
\end{align}
\end{lemma}
\begin{proof}
Since ${\cal F}_\gamma(\theta')=
\theta' - \frac{1}{\gamma} \Omega[\theta']$, we have
\begin{align}
&J _\gamma(\theta,\theta') \nonumber\\
=&\gamma \Big(\sum_{d=1}^d \eta_j(\theta)(\theta^j- {\theta'}^j + \frac{1}{\gamma} \Omega^j(\theta')) 
- \phi(\theta)+\phi(\theta') \Big) \nonumber\\
=&\gamma \sum_{d=1}^d \eta_j(\theta)(\theta^j-{\cal F}_\gamma^j(\theta'))
- \phi(\theta)
+ \phi({\cal F}_\gamma(\theta'))
\nonumber\\&+
\phi(\theta') - \phi({\cal F}_\gamma(\theta'))
\Big) \nonumber\\
=&\gamma \Big(D^\phi(\theta\| {\cal F}_\gamma(\theta'))
+ \phi(\theta') - \phi({\cal F}_\gamma(\theta'))
\Big)\nonumber\\
=&\gamma \Big(
D^\phi(\theta\| \Pro^{(e),\phi}_{\mathcal{M}} \circ{\cal F}_\gamma(\theta'))
+D^\phi(\Pro^{(e),\phi}_{\mathcal{M}} \circ{\cal F}_\gamma(\theta')\| {\cal F}_\gamma(\theta'))\nonumber\\
&+ \phi(\theta') - \phi({\cal F}_\gamma(\theta'))
\Big),
\Label{ASS4}
\end{align}
where the final equation follows from
the relation:
\begin{align}
D^\phi(\theta\|{\cal F}_\gamma(\theta'))
=&
D^\phi(\theta\| \Pro^{(e),\phi}_{\mathcal{M}} \circ{\cal F}_\gamma(\theta'))
\nonumber\\
&+D^\phi(\Pro^{(e),\phi}_{\mathcal{M}} \circ{\cal F}_\gamma(\theta')\|{\cal F}_\gamma(\theta')),
\Label{NM9}
\end{align}
which is Eq. \eqref{BVY} with $\overline{\theta}={\cal F}_\gamma(\theta')$.
Since only the term 
$D^\phi(\theta\| \Pro^{(e),\phi}_{\mathcal{M}} \circ{\cal F}_\gamma(\theta'))
$ depends on $\theta$  in \eqref{ASS4},
the minimum 
$\min_{\theta \in \Theta}  J_\gamma(\theta,\theta')$
is given as \eqref{XMY}, and it is realized with 
$\Pro^{(e),\phi}_{\mathcal{M}} \circ{\cal F}_\gamma(\theta')$.

Applying \eqref{XMY} into the final line of \eqref{ASS4},
we obtain \eqref{XMY2UU}.
Since the minimum in \eqref{XMY2UU} is realized when 
$\tilde{\theta}=
\Pro^{(e),\phi}_{\mathcal{M} \circ{\cal F}_\gamma(\theta')}$, 
we obtain \eqref{XMY2}.
\end{proof}

As a generalization of another part of \cite[Lemma 3.2]{RISB}, we can calculate 
$\argmin_{\theta' \in \Theta}  J _\gamma(\theta,\theta')$ as follows.
\begin{lemma}\Label{L2}
Assume that two elements $\theta,\theta' \in \Theta$ satisfy the 
condition \eqref{BK1+}.
Then, we have $\theta=\argmin_{\theta' \in \Theta}  J _\gamma(\theta,\theta')$, i.e., 
\begin{align}
J _\gamma(\theta,\theta')\ge J _\gamma(\theta,\theta) .
\end{align}
\end{lemma}
\begin{proof}
Eq. \eqref{BK1+} guarantees that
\begin{align}
&J _\gamma(\theta,\theta')-J _\gamma(\theta,\theta) 
=J _\gamma(\theta,\theta')-{\cal G}(\theta) \nonumber \\
=&
\gamma D^\phi(\theta\|\theta')+
\sum_{j=1}^d \eta_j(\theta) \Omega^j(\theta')-{\cal G}(\theta) \nonumber \\
=&\gamma D^\phi(\theta\|\theta')-D_\Omega(\theta \|\theta')
\ge 0 \Label{XMY5}.
\end{align}
\end{proof}

Now, we prove Theorem \ref{NLT}.
when all pairs 
$(\theta,\theta')=(\theta^{[t]},\theta^{[t+1]})$
satisfies \eqref{BK1+}, 
the relations 
\begin{align}
{\cal G}(\theta^{[t]}) &=J _\gamma(\theta^{[t]},\theta^{[t]})
\stackrel{(a)}{\ge}
J _\gamma(\theta^{[t+1]},\theta^{[t]}) \nonumber\\
& \stackrel{(b)}{\ge} J _\gamma(\theta^{[t+1]},\theta^{[t+1]})= {\cal G}(\theta^{[t+1]})
\Label{SAC}
\end{align}
hold under Algorithm \ref{AL1}, 
where $(a)$ follows from \eqref{XMY} of Lemma \ref{L1}
and $(b)$ follows from Lemma \ref{L2}.
Thus, Algorithm \ref{AL1} 
always iteratively improves the value of the objective function.
Thus, when the minimum of ${\cal G}(\theta)$ exists,
the relation \eqref{SAC} guarantees that 
the sequence $\{{\cal G}(\theta^{[t]})\}$ converges.

\section{Proof of Theorem \ref{TH1}}\Label{S10}
\subsection{Preparation for proof of Theorem \ref{TH1}}
To show Theorem \ref{TH1}, we prepare the following lemma.
\begin{lemma}\Label{LLX}
For any density matrices $\theta,\theta' \in \Theta$, we have
\begin{align}
&D^\phi(\theta\| \theta' )- D^\phi(\theta\| \Pro^{(e),\phi}_{\mathcal{M}} \circ{\cal F}_\gamma(\theta'))  \nonumber\\
=&\frac{1}{\gamma}J_\gamma(\Pro^{(e),\phi}_{\mathcal{M}} \circ{\cal F}_\gamma(\theta'),\theta') 
-\frac{1}{\gamma}{\cal G}(\theta)
+\frac{1}{\gamma}
D_\Omega(\theta\|\theta')
\Label{XM1} \\
=&\frac{1}{\gamma}{\cal G}(\Pro^{(e),\phi}_{\mathcal{M}} \circ{\cal F}_\gamma(\theta')) -\frac{1}{\gamma}{\cal G}(\theta) 
+
D^\phi(\Pro^{(e),\phi}_{\mathcal{M}} \circ{\cal F}_\gamma(\theta')\|\theta')
\nonumber \\
&-\frac{1}{\gamma} D_\Omega(\Pro^{(e),\phi}_{\mathcal{M}} \circ{\cal F}_\gamma(\theta') \|\theta')
+\frac{1}{\gamma}D_\Omega(\theta\|\theta')
.\Label{XM2}
\end{align}
\end{lemma}

\begin{proof}
We have
\begin{align}
-\sum_{d=1}^d \eta_j(\theta)\Omega^j(\theta')
=-{\cal G}(\theta)+ D_\Omega(\theta\|\theta').
\Label{MLP}
\end{align}
Using \eqref{MLP}, we have
\begin{align}
&D^\phi(\theta\| \theta')
- D^\phi(\theta\| \Pro^{(e),\phi}_{\mathcal{M}} \circ{\cal F}_\gamma(\theta'))
  \nonumber\\
=&\sum_{d=1}^d \eta_j(\theta)
(\theta^j- {\theta'}^j)
+\phi(\theta') - \phi(\theta)
\nonumber \\
&
- D^\phi(\theta\| \Pro^{(e),\phi}_{\mathcal{M}} \circ{\cal F}_\gamma(\theta'))
\nonumber \\
=&\sum_{d=1}^d \eta_j(\theta)
(\theta^j-{\cal F}_\gamma^j(\theta')+{\cal F}_\gamma^j(\theta')- {\theta'}^j)
+\phi({\cal F}_\gamma(\theta'))
\nonumber \\
&
 - \phi(\theta)
+\phi(\theta') 
- \phi({\cal F}_\gamma(\theta'))
- D^\phi(\theta\| \Pro^{(e),\phi}_{\mathcal{M}} \circ{\cal F}_\gamma(\theta'))
\nonumber \\
\stackrel{(a)}{=}&
D^\phi(\theta\|{\cal F}_\gamma(\theta'))
-\frac{1}{\gamma}\sum_{d=1}^d \eta_j(\theta)\Omega^j(\theta')
\nonumber \\
&+\phi(\theta') 
- \phi({\cal F}_\gamma(\theta'))
- D^\phi(\theta\| \Pro^{(e),\phi}_{\mathcal{M}} \circ{\cal F}_\gamma(\theta'))
\nonumber \\
\stackrel{(b)}{=}&
D^\phi(\Pro^{(e),\phi}_{\mathcal{M}} \circ{\cal F}_\gamma(\theta')\|{\cal F}_\gamma(\theta'))\nonumber \\
&-\frac{1}{\gamma}\sum_{d=1}^d \eta_j(\theta)\Omega^j(\theta')
+\phi(\theta') 
- \phi({\cal F}_\gamma(\theta'))
\nonumber \\
\stackrel{(c)}{=}&
\frac{1}{\gamma}J_\gamma(\Pro^{(e),\phi}_{\mathcal{M}} \circ{\cal F}_\gamma(\theta'),\theta') 
-\frac{1}{\gamma}{\cal G}(\theta)
+\frac{1}{\gamma} D_\Omega(\theta\|\theta')\Label{VB1}\\
\stackrel{(d)}{=}&
\frac{1}{\gamma}{\cal G}(\Pro^{(e),\phi}_{\mathcal{M}} \circ{\cal F}_\gamma(\theta')) -\frac{1}{\gamma}{\cal G}(\theta)
+D^\phi(\Pro^{(e),\phi}_{\mathcal{M}} \circ{\cal F}_\gamma(\theta')\|\theta')
\nonumber\\
&-\frac{1}{\gamma}D_{\Omega}
(\Pro^{(e),\phi}_{\mathcal{M}} \circ{\cal F}_\gamma(\theta')\|\theta')
+\frac{1}{\gamma}D_\Omega(\theta\|\theta'),\Label{VB2}
\end{align}
where each step is shown as follows.
$(a)$ follows from the definition of ${\cal F}_\gamma$.
$(c)$ follows from \eqref{XMY} and \eqref{MLP}. 
$(d)$ follows from \eqref{XMY5}. 
$(b)$ follows from \eqref{NM9}.
Then, \eqref{VB1} and \eqref{VB2} show 
\eqref{XM1} and \eqref{XM2}, respectively.
\end{proof}

\subsection{Proof of Theorem \ref{TH1}}\Label{S10}
\noindent{\bf Step 1:}\quad
The aim of this step is to show the following inequality;
\begin{align}
D^\phi(\theta_{*}\| \theta^{[t]})- D^\phi(\theta_{*}\| \theta^{[t+1]}) \ge \frac{1}{\gamma}{\cal G}(\theta^{[t+1]}) -\frac{1}{\gamma}{\cal G}(\theta_{*})\Label{XPZ2}
\end{align}
for $t=1, \ldots, t_0-1$.
We show these relations by induction.

For any $t$, 
by using the relation 
$ {\cal F}_\gamma(\theta^{[t]})=\theta^{[t+1]}$, 
the application of \eqref{XM2} of Lemma \ref{LLX} to the case with 
$\theta'=\theta^{[t]}$ and $\theta=\theta_{*}$
yields
\begin{align}
&D^\phi(\theta_{*}\| \theta^{[t]})- D^\phi(\theta_{*}\| \theta^{[t+1]})  \nonumber\\
=&\frac{1}{\gamma}{\cal G}( \theta^{[t+1]}) 
-\frac{1}{\gamma}{\cal G}(\theta_{*}) 
+
D^\phi(\Pro^{(e),\phi}_{\mathcal{M}} \circ
{\cal F}_\gamma(\theta^{[t]})\|\theta^{[t]})\nonumber\\
&-\frac{1}{\gamma}D_\Omega(
\Pro^{(e),\phi}_{\mathcal{M}} \circ{\cal F}_\gamma(\theta^{[t]}) \|\theta^{[t]})
+\frac{1}{\gamma}D_\Omega(\theta_{*}  
\|\theta^{[t]}) \\
=&\frac{1}{\gamma}{\cal G}( \theta^{[t+1]}) 
-\frac{1}{\gamma}{\cal G}(\theta_{*}) 
+D^\phi(\theta^{[t+1]}\|\theta^{[t]})
\nonumber\\
&-\frac{1}{\gamma}D_\Omega(\theta^{[t+1]} \|\theta^{[t]})
+\frac{1}{\gamma}D_\Omega( \theta_{*} \|\theta^{[t]})
.\Label{XME2}
\end{align}
Since two densities ${\cal F}_\gamma(\theta^{[t]})$ and 
$\theta^{[t]}$ satisfy the conditions \eqref{BK1+} and \eqref{BK2+},
we have
\begin{align}
\hbox{\rm(RHS of \eqref{XME2})} 
\ge 
\frac{1}{\gamma}{\cal G}(\theta^{[t+1]}) -\frac{1}{\gamma}{\cal G}(\theta_{*}).\Label{AXE}
\end{align}
The combination of \eqref{XME2} and \eqref{AXE} implies \eqref{XPZ2}.

\noindent{\bf Step 2:}\quad
This step aims to show \eqref{XME}.
Lemmas \ref{L1} and \ref{L2}
guarantee that
\begin{align}
{\cal G}(\theta^{[t+1]}) \le {\cal G}(\theta^{[t]}) \Label{AMK}.
\end{align}
We have
\begin{align}
& \frac{t_0}{\gamma} \Big(
{\cal G}(\theta^{[t_0+1]}) - {\cal G}(\theta_{*}) \Big)
\stackrel{(a)}{\le}  \frac{1}{\gamma}\sum_{t=1}^{t_0}
{\cal G}(\theta^{[t+1]}) -{\cal G}(\theta_{*}) 
\nonumber\\
\stackrel{(b)}{\le}& \sum_{t=1}^{t_0}
D^\phi(\theta_{*}\| \theta^{[t]})- D^\phi(\theta_{*}\| \theta^{[t+1]})  
\nonumber\\
=& 
D^\phi(\theta_{*}\| \theta^{[1]})-
D^\phi(\theta_{*}\| \theta^{[t_0+1]})  
\le D^\phi(\theta_{*}\| \theta^{[1]}),
\end{align}
where $(a)$ and $(b)$ follow from \eqref{AMK} and \eqref{XPZ2}, respectively.

\begin{remark}
When the condition \eqref{BK2+} does not hold,
the above proof does not work.
However, when 
$D^\phi(\theta^{[t+1]}\|\theta^{[t]})
-\frac{1}{\gamma}D_\Omega(\theta^{[t+1]} 
\|\theta^{[t]})
+\frac{1}{\gamma}D_\Omega(\theta_{*}  
\|\theta^{[t]})\ge 0$,
the above proof does work.
Maybe, there is a possibility 
that this proof locally works with a sufficiently large number $\gamma$.
\end{remark}


\section{Discussion}\Label{S11}
We have generalized the algorithms by 
\cite{RISB,Hmixture,HL}
by using the concept of Bregman divergence, which is a key concept of information 
geometry.
While the existing generalized Arimoto-Blahut algorithm \cite{RISB,HL,Hmixture}
works with a general setting,
their objective function needs to be defined over a set of probability distributions or quantum states.
We have removed this restriction, and have extended their algorithm to the setting with Bregman divergence.
When our method is applied to the case with 
probability distributions or quantum states,
we are allowed to choose the Bregman divergence
as a divergence different from 
the KL divergence or quantum relative entropy. 

Indeed, the 
existing methods \cite{Hmixture,HL} require to calculate
$e$-projection, which requires a convex minimization and can be considered as the bottleneck in the algorithm.
Choosing the Bregman divergence as a different divergence from the actual divergence
in our general algorithm,
we have proposed a minimization-free-iteration
iterative algorithm 
for the general problem studied in \cite{Hmixture,HL}.
The existing method in \cite{Hmixture,HL} covers
the em-algorithm and the derivation of the optimal conditional distribution for the 
rate-distortion theory.
We have applied our minimization-free-iteration algorithm to these problems.
In particular, as a special case of the em-algorithm,
we have numerically applied 
our obtained algorithm to the rate-distortion theory.
Since our algorithm has no convex minimization in each iteration,
our algorithm has a smaller number of iterations than 
the existing algorithm presented in \cite{H23}
when we count the number of iterations in convex minimization in 
the algorithm presented in \cite{H23}.
Therefore, 
it is an interesting future problem to apply our method 
to the problem of the em-algorithm, i.e.,
the minimization of the divergence between a mixture family and an exponential family,
Indeed, since the em-algorithm can be used for 
graphical model \cite{Lauritzen},
it is expected that this research direction has a wider applicability in machine learning.

When the objective function is a convex function, we have shown that the iteration of our algorithm coincides with the iteration of the mirror descent method.
Although this fact was shown by \cite{HSF1}
for the case discussed in \cite{RISB},
 this fact had been an open problem for a more general case studied in  \cite{Hmixture,HL}.

Our general framework can be applied to any function function with the form \eqref{form2}.
Although we have mainly discussed 
a minimization-free-iteration algorithm 
when the objective function is given over mixture family of probability distributions or quantum states,
the idea in Section \ref{S5} can be extendable to more general cases
as follows.
Once the optimization problem is given 
by a mixture parameter $\eta$ in the form \eqref{form2},
we choose a natural parameter $\theta$ to satisfy 
\eqref{B6K}.
Then, we can apply the discussion given in \eqref{form9}--\eqref{BNB}.
It is an interesting future problem to 
apply this idea to a more general class of objective functions
because this method works with 
the modification \eqref{form9}--\eqref{BNB} of the objective function.

\end{document}